\documentclass[11pt]{amsart}
\usepackage{amssymb,latexsym}
\newtheorem{theorem}{Theorem}[section]
\newtheorem{prop}[theorem]{Proposition}
\newtheorem{lemma}[theorem]{Lemma}
\newtheorem{remark}[theorem]{Remark}
\newtheorem{question}[theorem]{Question}
\newtheorem{definition}[theorem]{Definition}
\newtheorem{cor}[theorem]{Corollary}

\newtheorem{conj}[theorem]{Conjecture}

\begin{document}

\title[Anti-invariant cohomology of almost complex 4-manifolds]
{On the $J$-anti-invariant cohomology of almost complex 4-manifolds}
\author{Tedi Draghici}
\address{Department of Mathematics\\Florida International Univ.\\Miami, FL 33199}
\email{draghici@fiu.edu}
\author{Tian-Jun Li }

\address{School  of Mathematics\\  University of Minnesota\\ Minneapolis, MN 55455}
\email{tjli@math.umn.edu}
\author{Weiyi Zhang}
\address{Department of Mathematics\\  University of Michigan\\ Ann Arbor, MI
48109}
\email{wyzhang@umich.edu}

\begin{abstract} For a compact almost complex 4-manifold $(M,J)$, we study
the subgroups $H^{\pm}_J$ of $H^2(M, \mathbb{R})$ consisting of cohomology
classes representable by $J$-invariant, respectively, $J$-anti-invariant 2-forms.
If $b^+ =1$, we show that for generic almost complex structures on $M$, the subgroup
$H^-_J$ is trivial. Computations of the subgroups and their dimensions
$h^{\pm}_J$ are obtained for almost complex structures related to integrable ones.
We also prove semi-continuity properties for $h^{\pm}_J$. 
\end{abstract}

\maketitle
\section{Introduction}

For any almost complex manifold $(M, J)$, the last two authors
\cite{LZ} introduced certain subgroups of the de Rham cohomology
groups, naturally defined by the almost complex structure. These
subgroups are interesting almost
complex invariants and there are several works already devoted to
their study \cite{FT}, \cite{AT}, \cite{DLZ}, \cite{AT2}. Particularly important
are the subgroups $H_J^+$, $H_J^-$ of $H^2(M, \mathbb{R})$,
defined as the sets of cohomology classes which can be represented
by $J$-invariant, respectively, $J$-anti-invariant real $2-$forms.
They naturally appear in the relationship
between the compatible symplectic cone and the tamed symplectic
cone of a given compact almost complex manifold \cite{LZ}. All of the above
quoted works consider the problem of whether or not the subgroups
$H_J^+$, $H_J^-$ induce a direct sum decomposition of $H^2(M,
\mathbb{R})$. This is known to be true for integrable
almost complex structures $J$ which admit compatible K\"ahler
metrics on compact manifolds of any dimension. In this case, the
induced decomposition is nothing but the classical (real)
Hodge-Dolbeault decomposition of $H^2(M, \mathbb{R})$ (see
\cite{LZ}, \cite{FT}, \cite{DLZ}). On the other hand, examples from
\cite{FT}, \cite{AT}, \cite{AT2} show that there exist almost complex
structures, even integrable ones, on compact manifolds of dimension
greater or equal to 6, for which the subgroups $H_J^+$,
$H_J^-$ may even have a non-trivial intersection. Dimension 4 is
special, as it was proved in \cite{DLZ} that on any compact almost
complex 4-manifold $(M^4, J)$, the subgroups $H_J^+$, $H_J^-$
do yield a direct sum decomposition for $H^2(M, \mathbb{R})$. In
this paper, we still concentrate to dimension 4, and give some computations and
estimates for the dimensions $h_J^{\pm}$ of the subgroups
$H_J^{\pm}$.

\vspace{0.2cm}

After some preliminaries, section 2 contains a result of Lejmi
\cite{Le} (Lemma \ref{mehdi} in our paper), from which one can see
the space $H_J^-$ as the kernel of an elliptic operator. Following
an observation of Vestislav Apostolov, Lejmi's lemma combined with a
classical result of Kodaira and Morrow yields semi-continuity
properties of $h_{J_{t}}^{\pm}$ for any path $J_t$ of almost complex
structures on a compact 4-manifold (Theorem \ref{semicont-hJt}).
Two conjectures are made about the dimension $h^-_J$ on a compact 4-manifold:
namely, that $h^-_J$ vanishes for generic almost complex structures (Conjecture \ref{conj1}),
and that an almost complex structure with $h^-_J \geq 3$ is necessarily
integrable (Conjecture \ref{conj2}).

\vspace{0.2cm}

In section 3, we confirm the first conjecture for 4-manifolds
with $b^+=1$ (Theorem \ref{int3}). The main topic of section 3 is
the notion of metric related almost complex structures. Two almost
complex structures are said to be {\it metric related} if they
induce the same orientation and they admit a common compatible
metric. We compute the subgroups $H_J^+$, $H_J^-$ and their
dimensions $h^+_J$, $h^-_J$ for almost complex structures metric
related to an integrable one.  The main result is:

\begin{theorem} \label{hJ-}
Let $(M,J)$ be a compact complex surface. If $\tilde J$ is an almost
complex structure on $M$ metric related to $J$, $\tilde J \not\equiv
\pm J$, then $h^-_{\tilde J} \in \{ 0, 1, 2 \}$. The almost complex
structures $\tilde J$ with $h^-_{\tilde J} = 0$ form an open and
dense set with respect to the $C^{\infty}$-topology in the space of
almost complex structures metric related to $J$. The almost complex
structures ${\tilde J}$ for which $h^-_{\tilde J} = 1$ or
$h^-_{\tilde J} = 2$ are explicitly described. In particular, the
case $h^-_{\tilde J} = 2$ appears only when $(M,J)$ is a complex
torus, or a K3 surface.
\end{theorem}

A main tool in the proof of Theorem \ref{hJ-} are the Gauduchon metrics. 
We also use them to give an alternative proof of the fact first observed in \cite{LZ},
that for a compact complex surface $(M,J)$, $J$ is tamed by a symplectic form 
if and only if $b_1$ is even (Proposition \ref{tamecx}).
Section 3 ends with a couple of applications of Theorem \ref{hJ-}.
In Theorem \ref{nonpure}, we prove that the intersection of $H_J^+$,
$H_J^-$ could be non-trivial even for a K\"ahler $J$ if the
compactness assumption is removed. Theorem \ref{K3T4} shows that the
examples of non-integrable almost complex structures with $h_{\tilde
J}^-=2$ from Theorem \ref{hJ-} cannot admit a smooth
pseudo-holomorphic blowup.

\vspace{0.2cm}

The so called {\it well-balanced} almost
Hermitian 4-manifolds are introduced in section 4, as a natural generalization of both the
Hermitian 4-manifolds and the almost K\"ahler ones. It is likely
that this new notion has links with generalized complex geometry,
but we leave the study of these possible links for future work. For
now, we give some examples of well-balanced almost Hermitian
4-manifolds (Proposition \ref{expwb}), and prove a vanishing result for
$h_J^-$ on a well-balanced compact almost Hermitian 4-manifold with
Hermitian Weyl tensor (Theorem \ref{mainwb}).

\vspace{0.2cm}

Finally, in section 5 we discuss Donaldson's symplectic version of
the Calabi-Yau equation on 4-manifolds. We observe that his
technique based on the Implicit Function Theorem can also be used to
obtain a stronger semi-continuity property for $h_J^{\pm}$ near a $J$ which
admits a compatible symplectic form (Theorem \ref{deformation}).

\vspace{0.2cm} {\bf Acknowledgments:} We are very grateful to V.
Apostolov for pointing out Theorem \ref{semicont-hJt} and for other
valuable suggestions. We also thank D. Angella and A. Tomassini for
good discussions and for sending us the preprint \cite{AT2}, and V.
Tosatti for useful comments.

\section{Definitions and preliminary results}

Let $(M, J)$ be an almost complex manifold. The almost complex
structure $J$ acts on the bundle of real
2-forms $\Lambda^2$ as an involution,
\newline by $\alpha(\cdot, \cdot)
\rightarrow \alpha(J\cdot, J\cdot)$,  thus we have the splitting
into $J$-invariant,
respectively, $J$-anti-invariant 2-forms
\begin{equation} \label{formtype}
\Lambda^2=\Lambda_J^+\oplus \Lambda_J^-.
\end{equation}
We will denote by $\Omega^2$ the space of 2-forms on $M$
($C^{\infty}$-sections of the bundle $\Lambda^2$) , $\Omega_J^+$ the
space of $J$-invariant 2-forms, etc. For any $\alpha \in \Omega^2$,
the $J$-invariant (resp. $J$-anti-invariant) component of $\alpha$
with respect to the decomposition (\ref{formtype}) will be denoted
by $\alpha '$ (resp. $\alpha ''$). We will also use the notation $
\mathcal Z^2$ for the space of closed 2-forms on $M$ and $\mathcal
Z_J^{\pm} = \mathcal Z^2 \cap \Omega_J^{\pm}$ for the corresponding
projections.

The bundle $\Lambda^-_J$ inherits an almost complex structure, still
denoted $J$, by
$$\alpha \in \Lambda^-_J \; \rightarrow \; J\alpha\in \Lambda^-_J,
\mbox{ where } J\alpha(X,Y) = -\alpha(JX,Y) .$$
It is well known that when $J$ is integrable (in any dimension), we have
$$ \beta \in \mathcal Z_J^- \Leftrightarrow J\beta \in \mathcal Z_J^- .$$
Conversely (see e.g. \cite{Sal}), if $(M, J)$ is a
connected almost complex 4-manifold and there is a pair $\beta \in \mathcal Z_J^-, J\beta \in \mathcal Z_J^-$
($\beta$ not identically zero), then $J$ is integrable.

\vspace{0.1cm}
The following definitions were introduced in \cite{LZ} for an arbitrary almost complex manifold $(M,J)$.

\begin{definition}
(i) The $J$-invariant, respectively, $J$-anti-invariant cohomology subgroups $H_J^{+}$, $H_J^{-}$, are defined by
\begin{equation} \nonumber
H_J^{\pm}=\{ \mathfrak{a} \in H^2(M;\mathbb R) | \exists \; \alpha\in \mathcal
Z_J^{\pm} \mbox{ such that } [\alpha] = \mathfrak{a} \} \, ;
\end{equation}
(ii) $J$ is said to be {\it $C^{\infty}$-pure} if $H_J^+\cap H_J^-=\{0\}$;
\newline (iii) $J$ is said to be {\it $C^{\infty}$-full} if $H_J^+ + H_J^- = H^2(M;\mathbb
R)$;
\newline (iv) $J$ is {\it $C^{\infty}$-pure and full} if $H_J^+ \oplus H_J^- = H^2(M;\mathbb
R)$.
\end{definition}

\noindent As noted in the introduction, when $J$ is integrable and admits a compatible K\"ahler metric, or when
$(M,J)$ is a complex surface,
the subgroups $H_J^{\pm}$ are nothing but the (real) Dolbeault
cohomology groups (see \cite{DLZ}, \cite{AT}):
\begin{equation}\label{same} H_J^{+}=H_{\bar \partial}^{1,1}\cap H^2(M;\mathbb R),
\quad H_J^{-}=(H_{\bar \partial}^{2,0}\oplus H_{\bar
\partial}^{0,2})\cap H^2(M;\mathbb R).
\end{equation}
In these cases, there is a weight 2 formal Hodge decomposition (more
generally, this is true whenever the Fr\"ohlicher spectral sequence
degenerates at first step), so $J$ is $C^{\infty}$-pure and full.
For complex dimensions greater or equal to 3, there are known
examples of complex structures for which the Fr\"ohlicher spectral
sequence does not degenerate at first step. Recently, Angella and
Tomassini have also shown in \cite{AT} that Iwasawa manifold $X^6$ admits
complex structures which are not $C^{\infty}$-pure nor full.
Other interesting examples appear in \cite{AT2},
showing, in particular, that the notions of $C^{\infty}$-pure and $C^{\infty}$-full are not related.
The first 6-dimensional examples of (non-integrable) almost complex
nilmanifolds which are not $C^{\infty}$-pure nor full were given by
Fino and Tomassini \cite{FT}.

\vspace{0.2cm}

By contrast, in dimension 4, the following result was proved in \cite{DLZ}:

\begin{theorem} \label{pf-dim4}
If $M$ is a compact 4-dimensional manifold then any almost complex
structure $J$ on $M$ is $C^{\infty}$-pure and full, i.e.
\begin{equation} \label{purefull}
H^2(M;\mathbb R)=H_J^+ \oplus H_J^- \; .
\end{equation}
\end{theorem}

\vspace{0.2cm} \noindent We refer to \cite{DLZ} for the proof of
Theorem \ref{pf-dim4}. It is based on Hodge theory and the
particularity of dimension 4 stemming from the self-dual,
anti-self-dual decomposition induced by the Hodge operator $*_g$ of
a Riemannian metric $g$ on $M$:
\begin{equation} \label{sdasd}
\Lambda^2=\Lambda_g^+\oplus \Lambda_g^-.
\end{equation}
If the metric $g$ is compatible with the almost complex structure
$J$ and we let $\omega$ be the fundamental form defined by
$\omega(\cdot, \cdot)=g(J\cdot, \cdot)$, the decompositions (\ref{formtype}) and
(\ref{sdasd}) are related by
\begin{equation}\label{type-Jinv}
\Lambda_J^+=\underline{\mathbb R}(\omega)\oplus \Lambda_g^-,
\end{equation}
\begin{equation} \label{type-sdasd}
\Lambda_g^+ = \underline{\mathbb R}(\omega) \oplus \Lambda_J^-.
\end{equation}
In particular, any $J$-anti-invariant 2-form in 4-dimensions is
self-dual, thus any closed, $J$-anti-invariant 2-form is harmonic,
self-dual. This enables us to identify the space $H^-_J $ with
$\mathcal Z_J^{-}$, and further, with the set ${\mathcal{H}}^{+,
\omega^{\perp}}_g$ of harmonic self-dual forms pointwise orthogonal
to $\omega$. In fact, it is an observation of Lejmi \cite{Le} that
this space can be seen as the kernel of an elliptic operator defined
on $\Omega_J^-$.
\begin{lemma} \label{mehdi} (\cite{Le}, Lemma 4.1)
Let $(M^4, g, J, \omega)$ be a compact, almost Hermitian 4-manifold.
Consider the operator
\[P  : \Omega_J^- \rightarrow \Omega_J^- \; , \; \; P(\psi) = (d \delta^g \psi)'' ,\]
where $\delta^g$ is the codifferential with respect to the metric
$g$ and the superscript $''$ denotes the projection $ \Omega^2
\rightarrow \Omega_J^-$. Then $P$ is a self-adjoint strongly
elliptic linear operator with kernel the $g$-harmonic
$J$-anti-invariant 2-forms.
\end{lemma}

\noindent Lemma 4.1 in \cite{Le} is stated for almost K\"ahler
4-manifolds, but the reader can easily check that its proof does not
use the assumption that $\omega$ is closed. Indeed, since
$\Omega_J^- \subset \Omega_g^+$ and since the Riemannian Laplace
operator
\newline $\Delta^g = d \delta^g + \delta^g d$ preserves the
decomposition (\ref{sdasd}), note that for $\psi \in
\Omega_J^-$,
\[ P(\psi) = \frac{1}{2} \Delta^g \psi - \frac{1}{4} <\Delta^g \psi, \omega> \omega .\]
Then since $\psi$ and $\omega$ are pointwise orthogonal, a short
computation gives
\[ <\Delta^g \psi, \omega> = - 2 \delta^g (<\psi, \nabla \omega>) + <\psi, \Delta^g \omega> .\]
The right side contains clearly only one derivative in $\psi$, and
the lemma follows easily. Here and later in the paper, $\delta^g$ denotes the divergence operator, i.e. the
adjoint of $d$ with respect to the metric $g$.

\vspace{0.2cm}

Let us denote the dimension of $H_J^{\pm}$ by $h^{\pm}_J$, let $b_2$
be the second Betti number, and $b^{\pm}$ be the ``self-dual'',
resp. ``anti-self-dual'' Betti numbers of the 4-manifold $M$. By
Theorem \ref{pf-dim4} and the observations above, we have
\begin{equation} \label{h^+_+ h^-}
h^+_J + h^-_J = b_2 ;
\end{equation}
\begin{equation}\label{easyestimate}
h^{+}_J\geq b^-, \quad h^-_J\leq b^+.
\end{equation}

\noindent We propose the following two conjectures:
\begin{conj} \label{conj1}
For generic almost complex structures $J$ on a compact 4-manifold $M$, $h_J^- = 0$.
\end{conj}

\begin{conj} \label{conj2}
On a compact 4-manifold, if $h_J^- \geq 3$ then $J$ is integrable.
\end{conj}

\noindent In the case $b^+=1$, Conjecture \ref{conj1} is proved in
Theorem \ref{int3}. Theorem \ref{hJ-} is a further
partial answer and motivation for both conjectures.

\vspace{0.2cm}

We end this section by establishing a path-wise semi-continuity property for
$h^{\pm}_J$ on a compact 4-manifold. This result was pointed out
to the first author by Vestislav Apostolov.

\begin{theorem} \label{semicont-hJt}
Let $M$ be a compact 4-manifold and let $J_t$, $t \in [0,1]$ be a
smooth family of almost complex structures on $M$. Then $h^-_{J_t}$
(resp. $h^+_{J_t}$) is an upper-semi-continuous (resp.
lower-semi-continuous) function in $t$. That is, for any $t
\in [0,1]$ there exists $\epsilon > 0$ such that if $s \in [0,1]$,
$|s - t| < \epsilon$,
\[h_{J_s}^- \leq h_{J_{t}}^- \; , \; \; h_{J_s}^+ \geq h_{J_{t}}^+ .\]
\end{theorem}

\begin{proof} The statement about $h^-_{J_t}$ follows directly from Lemma \ref{mehdi} and
a classical result of Kodaira and Morrow showing the
upper-semi-continuity of the kernel of a family of elliptic
differential operators (Theorem 4.3 in \cite{KM}). The statement
on $h^+_{J_t}$ follows from Theorem \ref{pf-dim4}.
\end{proof}

\noindent The following immediate corollary sheds some light on Conjecture \ref{conj1}
and on the density statement in Theorem \ref{hJ-}.

\begin{cor} \label{hJ-=0} If $(M^4, J)$ is a compact almost complex manifold with $h^-_J = 0$ and
$J_t$ is a deformation of $J$, then for small $t$, $h^-_{J_t} = 0$.
\end{cor}

\begin{remark} {\rm In Theorem \ref{deformation}, we establish a stronger semi-continuity
property for $h_J^{\pm}$ near an almost complex structure which
admits a compatible symplectic form. Theorems \ref{semicont-hJt} and
\ref{deformation} are no longer true in higher dimension, as recent
examples of Angella and Tomassini imply (see Propositions 4.1, 4.3 and
Examples 4.2, 4.4 in \cite{AT2}). Note that their Example 4.2 shows
that the semi-continuity property fails in dimensions higher than 4,
even if one has a path of almost complex structures which are
$C^{\infty}$-pure and full.}
\end{remark}

\section{Computations of $h^{-}_J$}

\subsection{Generic vanishing of $h_J^-$ when $b^+=1$}
In this subsection we confirm Conjecture \ref{conj1} when $b^+=1$.

\begin{theorem}\label{int3}  Suppose $M$ is a compact $4-$manifold with
$b^+=1$ admitting almost complex structures. The almost complex
structures $J$ with $h^-_J = 0$ form an open and dense subset in the
set of all almost complex structures on $M$, with
$C^{\infty}$-topology.
\end{theorem}

\subsubsection{Topology of the space of almost complex structures} \label{topJs}

Let $\mathcal{J}=\mathcal{J}^\infty$ be the  space of $C^{\infty}$ almost complex structures.
Let us first describe the $C^{\infty}$ topology of $\mathcal J^{\infty}$.

It is well known that the space $\mathcal J^l$ of $C^l$ almost complex
structures has a natural separable Banach manifold structure via the $C^l$ norm (see
\cite{MS} for example).
The natural $C^\infty$ topology on $\mathcal J^{\infty}$ is induced by
the sequence of semi-norms $C^0, C^1,\cdots,C^l,\cdots$.
Locally, near a $C^{\infty}$ almost complex structure $J$,
$\mathcal{J}$ is a subspace of $\mathcal{J}^l$ with finer topology.

With the $C^{\infty}$ topology, $\mathcal{J}=\mathcal{J}^\infty$ is  a
Fr\'echet manifold.
A complete metric inducing the $C^{\infty}$  topology on it can  be
defined by
\begin{equation}\label{metric-top}
d(x,y)=\sum_{k=0}^\infty \frac{\|x-y\|_k}{1+\|x-y\|_k}\;2^{-k}.
\end{equation}
Here $\|\cdot \|_k$ represents the $C^k$ semi-norm on it.

\subsubsection{The space of $g-$compatible almost complex
structures}
To prove the density statement in Theorem \ref{int3} we need to consider the space of
almost complex structures compatible with a fixed Riemannian metric $g$.
This can be described as
the space of $g$-self-dual 2-forms $\omega$ satisfying $|\omega|^2_g
= 2$ point-wise on $M$ (equivalently, the space of smooth sections of
the twistor bundle associated to $(M,g)$).
The $C^{\infty}$-topology on this subspace corresponds to
$C^{\infty}$-topology on the space of 2-forms.

Suppose  we also fix a $g-$compatible pair $(J, \omega)$. Then  any
$g-$compatible almost complex structure corresponds to a $2-$form
\begin{equation} \label{ombeta}
\tilde \omega = f \omega + \beta, \; \mbox{with $\beta \in
\Omega^-_J$, $f \in C^{\infty}(M)$ so that $2f^2 + |\beta|^2 = 2$. }
\end{equation}
For us, the following variation will be useful, extending an idea
from \cite{Lee}. Suppose further  a section $\alpha \in \Omega^-_J$
is given.  One can define new $g$-compatible almost complex
structures as follows: pick smooth functions $f$ and $r$ on $M$ so
that the form
\begin{equation} \label{om(f,h,alpha)}
\tilde \omega = f \omega + r \alpha
\end{equation} satisfies
$|\tilde \omega|^2_g = 2$, and let $\tilde J$ be the almost complex
structure defined by $(g, \tilde \omega)$. Equivalently, $f$ and $r$
should satisfy the pointwise condition
\begin{equation} \label{squarenorm2}
2f^2 + r^2 |\alpha|_g^2 = 2 .
\end{equation}
For any $\alpha \in \Omega^-_J$, one can find such functions $f$ and
$r$. For instance, take $r$ to be small enough so that $r^2
|\alpha|_g^2 < 2$ everywhere; then $f$ is determined up to sign
by $f = \pm (1 - \frac{1}{2} r^2 |\alpha|_g^2)^{1/2}$. Junho Lee's
almost complex structures $J_{\alpha}$ (see \cite{Lee}) are obtained
for the specific choice \footnote{There is a factor ``2'' difference
in the convention for the norm of a two form between our paper and
\cite{Lee}. For us, if $(g,J, \omega)$ is a 4-dimensional almost
Hermitian structure, $|\omega|^2_g = 2$, whereas in \cite{Lee},
$|\omega|^2_g = 1$. This explains the apparent difference between
our $r$ and $f$ and those in Proposition 1.5 in \cite{Lee}.}
\begin{equation} \label{JLhf}
r = \frac{4}{2+|\alpha|^2} \; \mbox{ and } \; f =
\frac{2-|\alpha|^2}{2 + |\alpha|^2} \; .
\end{equation}
Note that we actually get a pair of almost complex structures
$J^{\pm}_{\alpha}$, as for the above choice of $r$, we have the sign
freedom in choosing $f$. Junho Lee defines these almost complex
structures on a K\"ahler surface $(M,J,g)$ and uses them as a tool
for an easier computation of the Gromov-Witten invariants.
Particularly important in his work are the almost complex structures
$J_{\alpha}$ corresponding to {\it closed} $\alpha$'s, i.e
$\alpha\in \mathcal{Z}^-_J$.

\vspace{0.1cm}

Another natural choice for $(r,f)$ is
\begin{equation} \label{hf}
r = \pm f = \pm \frac{\sqrt{2}}{\sqrt{2 + |\alpha|^2}} \; .
\end{equation}
This corresponds to almost complex structures that arise from the
forms $\pm \omega + \alpha$, conformally rescaled to satisfy the
norm condition.

\vspace{0.1cm}

Even more generally, given $\alpha$, we may choose $r$ so that $r^2
|\alpha|_g^2 \leq 2$, with equality at some points, but then at such
points we have to require the smoothness of the function $(1 -
\frac{1}{2} r^2 |\alpha|_g^2)^{1/2}$. Note also that if such points
exists, then we no longer have an ``up to sign choice'' for $f$
overall.

\vspace{0.1cm}

Finally, note that we can (and will) choose $r$ to satisfy
$r^2 |\alpha|_g^2 < 2$ and be supported on a small open set in $M$.
Then, for $f = (1 - \frac{1}{2} r^2 |\alpha|_g^2)^{1/2}$,
the new almost complex structure $\tilde J$ coincides with
$J$ outside the support of $r$.

\vspace{0.2cm}

\subsubsection{Proof of Theorem \ref{int3}}

\begin{proof} First we show the density.
Let $J$ be an almost complex structure on $M$. It follows from
\eqref{type-sdasd} that $h_J^- \in \{0, 1\}$. If $h^-_J = 0$,
Corollary \ref{hJ-=0} shows that in any neighborhood of $J$ there
are other almost complex structures $\tilde J$ with $h_{\tilde
J}^-=0$. If $h^-_J = 1$, let $\alpha \in \mathcal{Z}^-_J$,
normalized so that $\int_M \alpha^2 = 1$. Pick a $J$-compatible
metric $g$ and let $\omega$ be the fundamental form associated to
$(g, J)$. The form $\alpha$ is $g$-harmonic and point-wise
orthogonal to $\omega$. As in \eqref{om(f,h,alpha)}, let $\tilde
\omega = f \omega + r \alpha$, for some functions $f, r$ satisfying
\eqref{squarenorm2}, and define $\tilde J$, the almost complex
structure induced by $(g, \tilde \omega)$. If $r\not\equiv 0$, it is
clear that $h_{\tilde J}^-=0$, as $\tilde \omega$ is no longer
point-wise orthogonal to $\alpha$. On the other hand, we can choose
$r$ to be compactly supported on a small set, so $\tilde J$ can be
arbitrarily close to $J$.

For openness, we prove that the complement is closed. Let $J_k$ be a
sequence of almost complex structures with $h^-_{J_k} = 1$
converging to the almost complex structure $J$. 
Let $g$ be a $J$-compatible Riemannian
metric and let
$$g_k (\cdot, \cdot) = \frac{1}{2} (g(\cdot, \cdot) + g(J_k\cdot, J_k\cdot)) .$$
Clearly, $g_k$ is a Riemannian metric compatible with $J_k$ and
$(g_k, J_k)$ converges to $(g,J)$. Denote by $\Delta^k$ the
Hodge-DeRham Laplace operator associated to $g_k$ and by
$\mathbb{G}^k$ the Green operator associated to $\Delta^k$.

Let $\psi$ be a non-zero $g$-harmonic, self-dual two form, normalized so that
$\int_M \psi^2 = 1$ (up to sign, $\psi$ is unique with these properties, as $b^+ = 1$).

Consider the Hodge decomposition of the 2-form  $\psi$ with
respect to each of the metrics $g_k$.
$$\psi = (\psi - \mathbb{G}^k (\Delta^k \psi)) +
\mathbb{G}^k (\Delta^k \psi) = \psi_{h, k} + \psi_{ex, k} ,$$
where $\psi_{h, k} = \psi - \mathbb{G}^k (\Delta^k \psi)$
denotes the $g_k$-harmonic part of $\psi$ and $\psi_{ex, k} =
\mathbb{G}^k (\Delta^k \psi)$ is the $g_k$-exact part of $\psi$.
Since $g_k \rightarrow g$ and $\Delta^g \psi = 0$, this implies
$$\psi_{h, k} \rightarrow \psi \; \; \; , \; \; \psi_{ex, k} \rightarrow
0, \mbox{ as $k \rightarrow \infty$.} $$ Moreover, if
$(\psi_{k,h})^+$ denotes the $g_k$-self-dual part of
$\psi_{k,h}$, we have
$$(\psi_{k, h})^+ \rightarrow \psi  \; . $$
But since $b^+ = h^-_{J_k} = 1$, the $g_k$-harmonic, self-dual forms
$(\psi_{k, h})^+$ are $J_k$-anti-invariant. Since $J_k \rightarrow
J$, it follows that $\psi$ must be $J$-anti-invariant. Thus,
$h^-_J =1$.
\end{proof}

\begin{remark} \label{expFT} {\rm There exist compact almost complex 4-manifolds $(M, J)$ with
$b^+ = 1$ and $h^-_J = 1$. Proposition 6.1 of \cite{FT} contains one such example (see also
Proposition \ref{expwb} (iii) in this paper, where this example appears in a different context).
Note also that any such almost complex structure cannot be tamed by a symplectic form,
as a consequence of Theorem 3.3 of \cite{DLZ}.}
\end{remark}

\subsection{When $J$ is integrable} If $(M^4,J)$ is a compact complex surface, it follows from
(\ref{same}) that $h_J^{\pm}$ are the same as the dimensions of the corresponding
Dolbeault groups
\begin{equation} \label{Jint-hpmcx} h_J^+=h_{\bar\partial}^{1,1}, \quad
h_J^-=2h_{\bar\partial}^{2,0}.
\end{equation}
Together with the signature theorem (Theorem 2.7 in \cite{BPV}), we
get
\begin{equation} \label{Jint-hpm}
h^+_J  =\left\{ \begin{array}{ll}
b^- +1&\hbox{if $b_1$ even} \\
b^- &\hbox{if $b_1$ odd,} \end{array} \right. \quad h^-_J = \left\{
\begin{array}{ll}
b^+ -1&\hbox{if $b_1$ even} \\
b^+ &\hbox{if $b_1$ odd.}  \end{array} \right.
\end{equation}
It is a deep, but now well known fact that the cases $b_1$ even/odd
correspond to whether the complex surface $(M,J)$ admits or not a
compatible K\"ahler structure. We observe that there is a more direct
proof for the following weaker statement.

\begin{prop} \label{tamecx} Let $(M,J)$ be a compact complex surface.
The following are equivalent:

\vspace{0.1cm}

(i) $b_1$ is even; (ii) $b^+ = h^-_J + 1 = 2 h^{2,0}_{\bar \partial} +1$; (iii) $J$ is tamed.

\vspace{0.1cm}

\noindent Similarly, the following are equivalent:

\vspace{0.1cm}

(i') $b_1$ is odd; (ii') $b^+ = h^-_J = 2 h^{2,0}_{\bar \partial}$; (iii') $J$ is not tamed.
\end{prop}
\noindent An almost complex structure $J$ is said to be {\it tamed} if there exists a symplectic form $\omega$
such that $\omega(X, JX) > 0$ for any non-zero tangent vector $X$. The tame-compatible question of Donaldson
\cite{D} predicts that on a compact 4-manifold any tame almost complex structure $J$ admits, in fact, a {\it compatible}
symplectic form, that is, a symplectic form $\tilde \omega$, so that $\tilde \omega(\cdot, J\cdot)$ is a
Riemannian metric.

It was first observed in \cite{LZ} using a result of \cite{HL},
that on a compact complex surface the tame condition is equivalent with
$b_1$ even. Proposition \ref{tamecx} gives a different proof of this fact.
Assuming Kodaira's classification, the tame condition is thus
equivalent with the compatibility. As Donaldson
points out, a direct confirmation of the tame-compatible question would lead to a different proof of the
fact that $b_1$ even corresponds to a complex surface of K\"ahler type. At least in the case $b^+=1$,
the tame-compatible question is known to be a consequence of the symplectic Calabi-Yau problem, also introduced by
Donaldson in \cite{D} (see also \cite{W}, \cite{TWY}, \cite{TW-survey}, and section \ref{sCY} below).

\vspace{0.1cm}

A key tool in our proof of Proposition \ref{tamecx} are the {\it Gauduchon metrics} whose
definition we recall next.
\subsubsection{Gauduchon metric}

For an (almost) Hermitian manifold $(M,
g, J, \omega)$, the {\it Lee form} $\theta$ is defined by $\theta = J\delta^g
\omega$, or, equivalently in dimension 4, by $d\omega =
\theta \wedge \omega$. It is well known that $d\theta$ is a conformal invariant.
When $J$ is integrable, the case when $\theta$ is closed (exact)
corresponds to locally (globally) conformal K\"ahler metrics.
Obviously, Hermitian metrics with $\theta=0$ are, in fact, K\"ahler
metrics.


\begin{definition} A Hermitian metric such that the Lie form is
co-closed, i.e. $\delta^g \theta = 0$, is called a  {\it Gauduchon metric}
(or {\it standard Hermitian metric}, in the original terminology of \cite{gauduchon1}).
\end{definition}


\noindent  The existence and uniqueness (up to homothety) of a Gauduchon metric in each conformal
class is shown in \cite{gauduchon1}. The result is much more general; it does not require integrability,
nor restriction to dimension 4.
For us, the key property of a (Hermitian) Gauduchon metric in dimension 4 is the following:


\noindent \begin{prop} \label{gaud} (\cite{gauduchon2}) On a compact
complex surface $M$ endowed with a Gauduchon metric $g$,
the trace of a harmonic, self-dual form is a constant.
\end{prop}

\noindent For the  proof of Proposition \ref{gaud}, we refer the
reader to Lemma II.3 in \cite{gauduchon2} (see also \cite{ad-dga},
Proposition 3, for a slightly different argument). The Proposition
\ref{gaud} implies that for Hodge decomposition arguments, the
Gauduchon metrics behave quite like the K\"ahler ones. This simple
fact yields good consequences.

\subsubsection{Proof of Proposition \ref{tamecx}}

\begin{proof} As we mentioned already (and is easy to check), for a complex surface
 the groups $H^{\pm}_J$ are identified with the (real) Dolbeault groups as in \eqref{same}.
Using \eqref{easyestimate}, we thus have
$$b^+ \geq h^-_J = 2 h^{2,0}_{\bar \partial} .$$
We'll show (ii) $\Rightarrow$ (i) $\Rightarrow$ (iii) $\Rightarrow$ (ii).

It is well known that for any almost complex
4-manifold, $b_1 + b^+$ is odd, thus (ii) $\Rightarrow$ (i) is obvious.

Now assume (i), which is equivalent with $b^+$ odd, by the above observation. It follows that
$b^+ > h^-_J = 2 h^{2,0}_{\bar \partial}$. Choose a $J$-compatible conformal class
and let $g$ be the Gauduchon metric with total volume one in this class; denote by $\omega$ the fundamental 2-form
induced by $(g,J)$. Let $\psi$ be a non-trivial harmonic self-dual
2-form, whose cohomology class $[\psi]$ is cup-product
orthogonal to $H^-_J$ (such $\psi$ exists because $b^+ > h^-_J$). From
Proposition \ref{gaud} and \eqref{type-sdasd}, $\psi$ decomposes
as
\begin{equation} \label{thetadecomp}
\psi = a \, \omega +  \beta, \mbox{ with $a$ constant and } \beta
\in \Omega^-_J \; .
\end{equation}
The constant $a$ is non-zero, by the assumption
that $[\psi]$ is cup product orthogonal to $H^-_J$. This implies right away that
$\psi$ is symplectic (as $\beta$ is self-dual and point-wise orthogonal to $\omega$).
By eventually replacing $\psi$ by $-\psi$, we can assume also that
$a>0$, so $J$ is tamed (by $\psi$ or $-\psi$). Thus, we proved (i) $\Rightarrow$ (iii).

Next, suppose that $\psi$ is a symplectic form
that tames $J$. As pointed out in \cite{D},  $\mathbb{R} \, \psi + \Lambda^-_J$ is
is a 3-dimensional bundle on $M$, positive-definite with respect to the wedge pairing and the volume form
$\psi^2$. This induces a $J$-compatible conformal class. Let $g$ be the Gauduchon metric
in this class and denote again by $\omega$ the fundamental form of $(g,J)$.
The form $\psi$ is $g$-self-dual and closed, thus it is harmonic. Then relation \eqref{thetadecomp} holds,
with $a>0$, by the assumption that $\psi$ tames $J$. It follows that $[\psi] \not\in H^-_J$, thus
$b^+ > h^-_J$. Now assume that $\psi_1$ and $\psi_2$ are harmonic self-dual
2-forms, whose cohomology classes $[\psi_1]$, $[\psi_2]$ are cup-product
orthogonal to $H^-_J$. As above,
$$ \psi_1 = a_1 \, \omega +  \beta_1 \; , \; \; \psi_2 = a_2 \, \omega +  \beta_2 ,$$
with $a_1, a_2$ non-zero constants and $\beta_1, \beta_2 \in \Omega^-_J$.
But then $ a_2 \, \psi_1 - a_1 \, \psi_2 = a_2 \, \beta_1 - a_1 \, \beta_2 $
is $J$-anti-invariant and closed. Together with the assumptions that $[\psi_1]$, $[\psi_2]$ are cup-product
orthogonal to $H^-_J$, this can happen only if
$$ a_2 \, \psi_1 - a_1 \, \psi_2 \equiv 0 \; . $$
Thus $b^+ - h^-_J = 1$, so  (iii) $\Rightarrow$ (ii) is proved.

Remark that the proof shows that (i), (ii), (iii) are also equivalent to (iv) $b^+ > h^-_J$.
The equivalence of (i'), (ii'), (iii') is then the negation of the above.
\end{proof}


\subsection{Comparing metric related almost complex structures}
Notice that when $J$ is integrable the dimensions $h^{\pm}_J$ are
topological invariants. Such a property is certainly no longer true
for general almost complex structures. However, we are still able to
calculate the exact value of $h_J^{\pm}$ for almost complex
structures which are metric related to integrable ones.
To achieve
this we first derive some general results about metric related
almost complex structures.
\subsubsection{Estimates for $g-$related almost complex structures}
We again fix a Riemannian metric $g$.

\begin{definition} Suppose $J$ and
$\tilde J$ are  two almost complex structures  inducing the same
orientation on a 4-manifold $M$.  $J$ and $\tilde J$ are said to be
{\it $g-$related} if they are both compatible with  $g$.
\end{definition}

\noindent It is clear that if $g$ has this property, then so does any metric from its
conformal class. Also, if $J$ and $\tilde J$ are  $g-$related then
\[\Lambda^-_{J} + \Lambda^-_{\tilde J} \subset \Lambda_g^+,
\mbox{ and hence } H^-_{J} + H^-_{\tilde J} \subset \mathcal H_g^+
.\] Recall that since any closed $J$-anti-invariant form is
harmonic, self-dual, we can identify $H^-_J$ with $\mathcal{Z}^-_J$
and see it as a subspace of $\mathcal H_g^+$ (the space of harmonic,
self-dual forms).

\noindent The following  observation is the key for the computations of
$h^{\pm}_J$ we achieve in this section.

\begin{prop} \label{int1}
Suppose $J$ and $\tilde J$ are $g-$related almost complex structures
on a connected 4-manifold $M$, with $\tilde J \not\equiv \pm J$.
Then ${\rm dim} \;(H^-_{J} \cap H^-_{\tilde J}) \leq 1$.
\end{prop}

\begin{proof} Let $\omega$ and $\tilde \omega$ be the corresponding self-dual
2-forms. By assumption, the set
\[ U = \{ p \in M \; | J(p) \neq \pm \tilde J(p) \} = \{ p \; | {\rm dim} \;
( {\rm Span} \{\omega(p), \tilde \omega(p) \}) = 2 \} \] is a
non-empty open set in $M$. Without loss of generality we can assume
that $U$ is connected. Otherwise, we can make the reasoning below on
a connected component of $U$.

Assume $  H^-_{J} \cap H^-_{\tilde J} \neq \{ 0 \} $ and let
$\alpha_{1}, \alpha_{2} \in  \mathcal{Z}^-_{J} \cap
\mathcal{Z}^-_{\tilde J}=H^-_{J} \cap H^-_{\tilde J}$, not
identically zero. Let $U'$ be the open subset of $U$ where neither
$\alpha_{1}$ or $\alpha_{2}$ vanishes. $U' \neq \emptyset $ because
$\alpha_{1}$ and $\alpha_{2}$ are $g$(-self-dual)-harmonic forms,
thus they satisfy the unique continuation property. Since on $U'$,
${\rm Span}\{\omega,\tilde \omega\}$ is a 2-dimensional subspace of
$\Lambda_g^+ M$ and $\alpha_{1}, \alpha_{2}$ are both orthogonal to
this subspace, there exists $f \in C^{\infty}(U')$ such that $
\alpha_{2} = f \alpha_{1}$. Since $\alpha_{1}, \alpha_{2}$ are, by
assumption, both closed, it follows that $0 = df \wedge \alpha_{1}$.
But $\alpha_{1}$ is non-degenerate on $U'$ (it is self-dual,
non-vanishing). Thus $df = 0$, so $f = const.$ on $U'$. It follows
that $\alpha_{2} = const. \; \alpha_{1}$ on $U'$, but, by unique
continuation, this holds on the whole $M$.
\end{proof}

\begin{remark}\label{openset} {\rm The estimate in Proposition \ref{int1} is
sharp. Indeed, let $(M, g, J, \omega)$ be a connected almost
Hermitian 4-manifold, and assume that $\alpha \in \mathcal{Z}^-_{J}$
is not identically zero. Consider a $g$-compatible almost complex
structure ${\tilde J}$, arising from a self-dual 2-form
\begin{equation} \label{hj-=1}
\tilde \omega = f \omega + r J\alpha \;,
\end{equation}
where $f$ and $r$ are $C^{\infty}$-functions, so that
\begin{equation} \label{norm2}
|\tilde \omega|^2_g = 2f^2 + r^2 |\alpha|^2_g = 2 \; .
\end{equation}
By \eqref{type-sdasd} applied to $(g, \tilde J, \tilde \omega)$,
observe that $\alpha$ is $\tilde J-$anti-invariant. Hence, by
Proposition \ref{int1}, $ H^-_{J} \cap H^-_{\tilde J} = {\rm
Span}([\alpha])$.
 Conversely,   any $g$-compatible $\tilde J$
such that $[\alpha] \in H^-_{J} \cap H^-_{\tilde J}$ will have a
fundamental form $\tilde \omega$ given by (\ref{hj-=1}) at least on
the open dense set $M' = M \setminus \alpha^{-1}(0)$, with functions
$f, r \in C^{\infty}(M')$ satisfying (\ref{norm2}).}
\end{remark}

\noindent Observe that compactness is not needed for Proposition
\ref{int1} or Remark \ref{openset}. In the compact case, Proposition
\ref{int1} has the following easy consequence.
\begin{cor} \label{atmostone}
In the space of  almost complex structures compatible to a given metric $g$
on a compact 4-manifold,
there is at most one $J$ such that
\begin{equation}\label{atmost} h_J^-\geq \left\{ \begin{array}{ll} \frac{b^+
+3}{2} & \hbox{if $b^+$ is odd} \\ \frac{b^+ +2}{2} & \hbox{if $b^+$
is even}.
\end{array} \right.
\end{equation}
\end{cor}

\vspace{0.1cm}

\subsubsection{Metric related almost complex structures}
\vspace{0.2cm}

\begin{definition}  Two almost complex structures $J$ and $\tilde J$
on a 4-manifold $M$ are said to be {\it metric related} if they induce the same
orientation and are $g-$related for some Riemannian metric $g$ on $M$.
\end{definition}

\noindent If we fix a volume form $\sigma$ on $M$,
two almost complex structures $J$ and $\tilde J$ are metric related
if and only if there exists a 3-dimensional sub-bundle $\Lambda^+
\subset \Lambda^2 M$, positive definite with respect to the wedge
pairing and $\sigma$, such that $\Lambda^-_{J} \subset \Lambda^+ $,
$\Lambda^-_{\tilde J} \subset \Lambda^+ $. One important
difference versus the ``$g$-related'' condition for a fixed $g$ is that the
metric related condition is not transitive.
Because of this,  Corollary \ref{atmostone}, for instance, is not
automatically clear under just the metric related assumptions.
However, Proposition \ref{int1} clearly extends to the metric
related case. One immediate consequence is:
\begin{cor}\label{12}
Suppose $J$ and $\tilde J$ are metric-related almost complex
structures on a compact 4-manifold $M$, with $\tilde J \not\equiv
\pm J$.

\vspace{0.1cm}

(i) If $h_J^-=b^+$, then $h_{\tilde J}^-\leq 1$.

(ii) If $h_J^-=b^+-1$, then $h_{\tilde J}^-\leq 2$.
\end{cor}

\noindent From this Corollary, one obtains immediately the claim $h_{\tilde
J}^- \in \{0, 1, 2\}$ from the statement of Theorem \ref{hJ-}. The
results in the next subsections will be more specific about when
each case occurs.

\vspace{0.2cm}

\subsection{Proof of Theorem \ref{hJ-}} \label{metint} Throughout
this subsection, unless stated otherwise,  $J$ will denote a {\it
complex} structure on a compact 4-manifold $M$. Denote by $\mathcal
J$ the space of all (smooth) almost complex structures on $M$ and by
$\mathcal{J}_J$ the set of almost complex structures which are
metric related to the fixed $J$. On both spaces $\mathcal J$ and
$\mathcal{J}_J$ we consider the $C^{\infty}$-topology.
For reasons that will be apparent soon, it is best to divide the
proof into some cases depending on the type of the surface $(M,J)$.

\subsubsection{Surfaces of non-K\"ahler type, or of K\"ahler type but with non-trivial canonical bundle}
For these we have the following result.

\begin{theorem} \label{nontrivK}
Let $(M,J)$ be a compact complex surface of non-K\"ahler type, or a
compact complex surface of K\"ahler type, but with topologically
non-trivial canonical bundle. If $\tilde J \in \mathcal{J}_J$, $\tilde J \not\equiv
\pm J$, then either (i) $h^-_{\tilde J} = 0$, or (ii) $h^-_{\tilde
J} = 1$. Case (i) occurs for an open, dense set of almost complex
structures in $\mathcal{J}_J$. Case (ii) occurs precisely when there
exist $\alpha \in \mathcal{Z }^-_{J}$ such that $ H^-_{\tilde J} = {\rm Span}([\alpha])$, so these
$\tilde J$ appear as described in Remark \ref{openset}.
\end{theorem}

%



\begin{proof} First, we justify the statement
$h^-_{\tilde J} \in \{0,1\}$.  For a
complex surface of non-K\"ahler type,
this follows directly from  Corollary \ref{12}.
%
%
%
%
%
%
Now suppose that $(M,J)$ is a complex surface of K\"ahler
type with topologically non-trivial canonical bundle. Consider  the
conformal class of metrics compatible with both $J$ and $\tilde J$
and let $g$ be the Gauduchon metric with respect to $J$ in this
class. Let $\omega$ and $\tilde \omega$ denote the fundamental forms
of $(g,J)$ and $(g, \tilde J)$, respectively. They are related as in
\eqref{ombeta},
\[ \tilde \omega = f \omega + \beta, \; \mbox{with $\beta \in
\Omega^-_J$, $f \in C^{\infty}(M)$ so that $2f^2 + |\beta|^2 = 2$. }
\]
Suppose $h^-_{\tilde J} \neq 0$ and let $\psi \in \mathcal{Z}^-_{\tilde J}$,
not identically zero. Since $\psi$ is $g$-harmonic, from Proposition
\ref{gaud} it must be of the form $\psi= a \omega + \alpha$, with
$a$ constant and $\alpha \in \Omega^-_J$. The
pointwise condition $<\psi, \tilde \omega> = 0$ is equivalent to
$$ 2a f + <\alpha,\beta> = 0  \mbox{ everywhere on $M$.} $$ But $\beta$ (and $\alpha$)
must vanish somewhere on $M$, since the canonical bundle is
topologically non-trivial. At a point $p$ where $\beta(p) = 0$, we
have $f^2(p) = 1 \neq 0$, thus it follows that $a=0$. Thus $\psi =
\alpha$, but since $d \psi = 0$, it follows that $\psi = \alpha \in
\mathcal{Z}^-_J$. Hence, $H^-_{\tilde J} \subset H^-_J$. The statement
$h^-_{\tilde J} \in \{0,1\}$ follows now from Proposition
\ref{int1}. Note that we also proved the description of the case $h^-_{\tilde J} = 1$.
%

\vspace{0.2cm}

Next, we prove the density statement in Theorem \ref{nontrivK}.
This
follows from Corollary \ref{hJ-=0} and the following observation.

\begin{prop} \label{exph-=0}
Let $(M,J)$ be a compact complex surface as in Theorem \ref{nontrivK}.
If $\tilde J \in \mathcal{J}_J$ and $h^-_{\tilde J} \neq 0$, then there exists
$\tilde J' \in \mathcal{J}_J$, arbitrarily close to $\tilde J$, and with $h^-_{\tilde J'} = 0$.
\end{prop}

\begin{proof} Suppose first that the geometric genus vanishes.
For non-K\"ahler type, this means  $b^+ = 2h^{2,0}_{\bar \partial} = 0$,
so it follows from \eqref{type-sdasd} that $h^-_{\tilde J} =
0$ for any $\tilde J$ on $M$ (even not metric related to $J$). If
$(M,J)$ is of K\"ahler type and has zero geometric genus,
then $h^-_J = 2h^{2,0}_{\bar \partial} = 0$, so the first part of the proof of
Theorem \ref{nontrivK} shows that $h^-_{\tilde J} = 0$, for any $\tilde J \in \mathcal{J}_J$.

Suppose next that the geometric genus of $(M,J)$ does not vanish.
Consider first the case $\tilde J \not\equiv \pm J$.
From the first part of the proof of Theorem \ref{nontrivK}, the assumption
$h^-_{\tilde J} \neq 0$ implies that there exists $\alpha \in
\mathcal{Z}^-_J$ such that $H^-_{\tilde J} = {\rm Span}\{ [\alpha]
\}$. Moreover, there is a metric $g$ on $M$ compatible with both $J$
and ${\tilde J}$ so that the corresponding forms $\omega$ and
$\tilde \omega$ are related on $M' = M \setminus \alpha^{-1}(0)$ as
in (\ref{hj-=1}):
\[ \tilde \omega = f \omega + r J\alpha \;, \]
where $f$ and $r$ are $C^{\infty}$-functions on $M'$, satisfying the
norm condition (\ref{norm2}). Note that even if the above relation
is valid on the (open, dense) set $M'$, $\tilde \omega$ is defined
on the whole $M$. We deform $\tilde \omega$ as follows. Let $\tilde
r$ be a compactly supported function on a small open subset $U$ of
$M'$ and define
\[ \tilde \omega ' = \tilde f \tilde \omega + \tilde r \alpha \;, \]
where the function $\tilde f$ is chosen so that $|\tilde \omega '|^2
= 2$. Let $\tilde J'$ be the almost complex structure induced by
$(g, \tilde \omega ')$. We claim that $h^-_{\tilde J'} = 0$.

Indeed, if $h^-_{\tilde J'} \neq 0$, as in the proof of Theorem
\ref{nontrivK}, there exists $\beta \in \mathcal{Z}^-_J(M)$, so that
$H^-_{\tilde J'} = {\rm Span}\{ [\beta ]\}$. Moreover, there are
functions $h, q$ so that
\[ \tilde \omega' = h \omega + q J\beta , \]
on the open dense set $M''= M \setminus \beta^{-1}(0)$. On the other
hand, on $M'$ we have
\[ \tilde \omega ' = (\tilde f f) \omega + (\tilde f r) J\alpha +
\tilde r \alpha \; .\] It follows that on $M' \cap M''$, we have
\[ q J\beta = (\tilde f r) J\alpha +
\tilde r \alpha \; . \] Since $\tilde r$ is compactly supported on a
small subset in $M'$, it follows that $J\alpha$ and $J\beta$ are
conformal multiples of one another on a non-empty open set. By the
argument in the proof of Proposition \ref{int1}, it follows that
$\alpha$ and $\beta$ are (non-zero) scalar multiples of one another
on the whole $M$. Thus, $H^-_{\tilde J'} = {\rm Span}\{ [\alpha] \}$.
But, by construction, on the set where $\tilde r
\neq 0$, the form $\tilde \omega '$ is not point-wise orthogonal to
$\alpha$. Thus, $h^-_{\tilde J'} = 0$, as claimed.

\vspace{0.2cm}

In the case $\tilde J \equiv \pm J$, the argument is similar. We
have even larger freedom in considering the deformation. Let $\alpha
\in \mathcal{Z}^-_J$, and let $r_1$, $r_2$ be compactly supported on
disjoint open sets. Consider
\[ \tilde \omega ' = f \omega + r_1 \alpha + r_2 J\alpha, \]
where $f$ is chosen to fulfill the norm condition. As above, one can
show that $h^-_{\tilde J'} = 0$.
\end{proof}

\begin{remark} \label{nononzero} {\rm The first part of the argument above shows the following:
suppose $(M,J)$ is a compact complex surface of K\"ahler type
with vanishing geometric genus and topologically non-trivial canonical bundle.
Then for any $\tilde J \in \mathcal{J}_J$, $h^-_{\tilde J} = 0$.}
\end{remark}

\vspace{0.1cm}

Finally, the openness statement in Theorem \ref{nontrivK} follows
from:
\begin{prop} \label{seqJ}
With the notations and assumptions of Theorem \ref{nontrivK},
suppose $\tilde J_k$ is a sequence of almost complex structures
converging to $\tilde J$ (in the $C^{\infty}$-topology), with
$\tilde J_k, \tilde J \in \mathcal{J}_J$. If $h^-_{\tilde J_k} \neq
0$, then $h^-_{\tilde J} \neq 0$.
\end{prop}
\begin{proof} The assumption $h^-_{\tilde J_k} \neq 0$ and the earlier arguments
in the proof of Theorem \ref{nontrivK}, show that there exists
$\alpha_k \in \mathcal{Z}^-_J$, such that $H^-_{\tilde J_k} = {\rm
Span}( [\alpha_k])$. We can normalize $\alpha_k$ so that $[\alpha_k
] \cdot [\alpha_k] = 1$, where $\cdot$ denotes here the cup-product
of cohomology. Thus, as $\alpha_k$ is a sequence on the unit sphere
in $\mathcal{Z}^-_J$ which is a compact set (note that
$\mathcal{Z}^-_J$ is finite dimensional), we can extract a
subsequence, still denoted $\alpha_k$, which converges to $\alpha
\in \mathcal{Z}^-_J$. Obviously, $[\alpha] \neq 0$, as $[\alpha]
\cdot [\alpha] = 1$. Moreover, since $\tilde J_k \rightarrow \tilde
J$, $\alpha_k \rightarrow \alpha$, the relation
$$\alpha_k(\tilde J_k X, \tilde J_k Y) = - \alpha_k(X, Y) $$
implies
$$ \alpha(\tilde J X, \tilde J Y) = - \alpha(X, Y) .$$
Thus, $h^-_{\tilde J} \neq 0$.
\end{proof}

\noindent This also completes the proof of Theorem \ref{nontrivK}.
\end{proof}

\vspace{0.1cm}

\begin{remark} {\rm A similar argument to the one in Proposition \ref{seqJ}
yields the following result: given a metric $g$ on a compact 4-manifold $M$,
the set of $g$-compatible almost complex structures $\tilde J$ with $h^-_{\tilde J} =0$
is open in the set of all $g$-compatible almost complex structures.}
\end{remark}

\begin{remark} \label{junholee} {\rm Under the assumptions of Theorem \ref{nontrivK},
if $\alpha \in \mathcal{Z}^-_J$,
then  the almost complex structures $\tilde J$ defined by
(\ref{om(f,h,alpha)}) have $h^-_{\tilde J} = 1$, for any choice of
$(r, f)$ satisfying (\ref{squarenorm2}). In particular this is true for
Junho Lee's almost complex structures $J^{\pm}_{\alpha}$ defined by
(\ref{JLhf}).
Note that since $J$ is integrable, $\alpha + iJ\alpha$
is a holomorphic $(2,0)$ form on $M$, hence the zero set
$\alpha^{-1}(0)$ is a canonical divisor on $(M, J)$.}
\end{remark}


\begin{remark} {\rm If a compact 4-manifold $M$ admits a pair of
integrable complex structure $(J_1, J_2)$ which are metric related
then $M$ has a bi-Hermitian structure. The study of such structures has been
active recently, (see, for instance,
\cite{Hitchin} and the references therein), especially due to the link with
generalized K\"ahler geometry (\cite{Gualtieri}). An easy
consequence of Theorem \ref{nontrivK} is
the observation that a compact 4-manifold $M$ with $b^+ = 2$, or
$b^+ \geq 4$ does not admit a bi-Hermitian structure (compatible
with the given orientation). This is not new, as it is easily seen
from the classification results of \cite{AGG} and \cite{Ap}, that
manifolds admitting bi-Hermitian structures must have $b^+ \in \{ 0,
1, 3 \}$. }
\end{remark}


\subsubsection{Surfaces of K\"ahler type with  topologically trivial,
 but holomorphically non-trivial canonical bundle}

\begin{prop} \label{hyperelliptic} Suppose that $(M,J)$  is a complex surface of
K\"ahler type with topologically trivial, but holomorphically
non-trivial canonical bundle. Then for any almost complex structure
$\tilde J$ on $M$ (not even metric related to $J$), we have
$h^-_{\tilde J} \in \{0, 1\}$. The set of almost complex structures
with $h^-_{\tilde J} = 0$ is open and dense with respect to the
$C^{\infty}$-topology, both in $\mathcal J$ and $\mathcal J_J$.
\end{prop}

\begin{proof} Any such  surface is a hyperelliptic
surface. In this case $b^+=1$ and the claims follow from
\eqref{type-sdasd} and Theorem \ref{int3}.
\end{proof}

We wonder whether the result in Remark \ref{nononzero} still holds
in this case; in other words, is it still true that $h^-_{\tilde J}
= 0$ for any $\tilde J \in \mathcal{J}_J$?

\subsubsection{Surfaces with holomorphically trivial canonical bundle}
Even if the non-K\"ahler subcase is covered by
Theorem \ref{nontrivK}, it is worth considering it separately, as
the result takes a very simple form. Surfaces of non-K\"ahler type with holomorphically
trivial canonical bundles are Kodaira surfaces.

Thus, let $(M,J)$ be a Kodaira surface. We have $h_J^-=b^+=2$. Let $\Phi = \beta + i J\beta$ be a a nowhere
vanishing holomorphic $(2,0)-$form trivializing the canonical
bundle. The real and imaginary parts of $\Phi$, $\beta$ and $J\beta$
are both closed, nowhere vanishing $J$-anti-invariant forms. Suppose
that $g$ is a metric compatible with $J$ and let $\omega$ be the
corresponding non-degenerate form of $(g,J)$. The triple $\{\omega,
 \beta, J\beta\}$ is a pointwise orthogonal basis of the
rank $3$ bundle $\Lambda_g^+$. Thus, any almost complex structure
compatible with $g$ corresponds to a form
\begin{equation} \label{omfls}
\omega_{f,l,s} = f\omega+l\beta+sJ\beta,
\end{equation}
where the functions $f,l,s \in
C^{\infty}(M)$ satisfy $2f^2 + |\beta|^2(l^2+s^2) = 2$.
We denote the almost complex structure corresponding to $(g,
\omega_{f,l,s})$ by $J_{f,l,s}$. Every almost complex structure
metric related to $J$ can be obtained this way.
Since for a Kodaira surface $\mathcal H_g^+=H_J^- = {\rm Span}(\alpha,
J\alpha)$, the only possible self-dual harmonic forms are of
type
\[ a\, \beta+b\, J\beta, \quad \hbox{ where $a$ and $b$ are
constants}.\]
The only condition for this form lying in
$H_{J_{f,l,s}}^-$ is $al+bs=0$. Thus we have proved

\begin{prop} \label{kodsurface}
If $(M, g, J)$ is a Kodaira surface with a compatible metric $g$,
using the notations above,
\[h_{J_{f,l,s}}^-= 2 - rank({\rm Span}(l, s)) .\]
Clearly, $h_{J_{f,l,s}}^-= 0$ is the generic case, $
h_{J_{f,l,s}}^-= 2$ if and only if $l=s=0$, i.e. $\tilde J = \pm J$,
and $ h_{J_{f,l,s}}^-= 1$ if and only if the functions $l$ and $s$
are scalar multiples of each other, not both identically zero.
\end{prop}


Next, suppose that $(M,J)$ is a K\"ahler surface with
holomorphically trivial canonical bundle. Then $b^+=3$ and $(M, J)$
is a K3 surface or $4-$torus. As in the Kodaira surface case, let
$\Phi = \beta + i J\beta$ be a a nowhere vanishing holomorphic
$(2,0)-$form trivializing the canonical bundle. Consider a conformal
class of metrics compatible with $J$, and, in this class, let $g$ be
the Gauduchon metric, with $\omega$ being the associated form. As
above, denote by $J_{f,l,s}$ the almost complex structure
corresponding to form $\omega_{f,l,s}$ given by (\ref{omfls}). Every
almost complex structure metric related to $J$ is of the type
$J_{f,l,s}$ for some Gauduchon metric $g$ and for some functions $f,
l, s$.

The difference from the Kodaira surface case is that $b^+={\rm dim}
(\mathcal H_g^+)=3$, rather than $2$. As argued in Theorem
\ref{nontrivK}, any $g-$harmonic form has a constant inner product
with $\omega$. Let $\omega'$ be the unique $g$-self-dual harmonic
form with $<\omega', \omega>_g = 2$ and which is cup-product
orthogonal to $H^-_J = {\rm span} \{\beta, J\beta \}$. This is
written as
\begin{equation} \label{uv}
\omega' = \omega+ u\, \beta + v\,J\beta \; ,
\end{equation}
where $u, v$ are $C^{\infty}$-functions. They satisfy
\[\int_M u |\beta|^2 \; d\mu_g = \int_M v |\beta|^2 \; d\mu_g  = 0 \; ,\]
and a differential equation corresponding to $d\omega' = 0$. Thus,
any self-dual harmonic form is of type
\[c\; \omega'+ a\; \beta+ b\; J\beta,\] where
$a$,$b$,$c$ are constants. The only condition for this form to be in
$H_{J_{f,l,s}}^-$ is to be point-wise orthogonal to $\omega_{f,l,s}$
(see (\ref{omfls})). This amounts to
\[2cf'+al'+bs'=0, \]
where
\begin{equation}\label{'''} l'=l|\beta|^2, \quad
s'=s|\beta|^2\quad \hbox{ and } f'=2f+ul'+vs' \; .
\end{equation}
Therefore we have the following statement.

\begin{prop} \label{linear} Suppose $(M,J)$ is a K\"ahler surface with
holomorphically trivial canonical bundle. Let $\beta$ be a closed form trivializing
the canonical bundle. Consider a conformal class compatible with $J$
and let $g$ be the Gauduchon metric in this class. Let $\omega$ be
the associated form and let $J_{f,l,s}$ be the $g$-related almost
complex structure defined via (\ref{omfls}). Then
\[h_{J_{f,l,s}}^-=3 - rank({\rm Span } (f',l',s')),\] with $f',
l', s'$ as in \eqref{'''}. The case $h_{J_{f,l,s}}^- = 0$ is the generic situation, thus
the set of almost complex structures $\tilde J$ with $h^-_{\tilde J} = 0$ is dense in
$\mathcal{J}_J$.  The cases $h_{J_{f,l,s}}^- = 2$, $h_{J_{f,l,s}}^- = 1$ also occur.
\end{prop}

\noindent Note that $g$ is a hyperK\"ahler metric precisely when
 $|\beta|^2=2$ pointwise and in this case $\omega' = \omega$.

\begin{remark} \label{a}
{\rm We leave to the interested reader to check the computation that
$h_{J_{f,l,s}}^- = 2$ if
and only if
\[f = \pm (1 - k_1 u - k_2 v) |\beta| w,  \; \; l = \pm 2k_1 |\beta|^{-1} w, \; \; s = \pm 2k_2 |\beta|^{-1} w, \]
where $k_1, k_2$ are arbitrary constants, $u, v$ are given by
(\ref{uv}), and
\[w = [(1 - k_1 u - k_2 v)^2 |\beta|^2 + 2(k_1^2 +
k_2^2)]^{-1/2} \; . \]
We just observe that most of the examples with $h_{J_{f,l,s}}^- = 2$ described above
are non-integrable almost complex structures. This can be again
checked directly, or one can argue as follows. If for a certain
metric $g$ and functions $f,l,s$, we obtain an {\it integrable}
almost complex structure $J_{f,l,s}$, then $(g, J, J_{f,l,s})$ is a
bi-Hermitian structure. It is well known that conformal classes
carrying bi-Hermitian structures are very particular, as Theorem 2
in \cite{AGG} shows. On the other hand, our Proposition \ref{linear}
shows that examples of almost complex structures $J_{f,l,s}$ with
$h^-_{J_{f,l,s}} = 2$ occur in {\it each} conformal class associated
to the given $J$. Thus, most of these $J_{f,l,s}$ must be
non-integrable. }
\end{remark}

\noindent As an extension of Conjecture \ref{conj2}, it is natural
to ask:
\begin{question}\label{geq2}
Are there (compact, 4-dimensional) examples of non-integrable almost
complex structures $J$ with $h^-_J \geq 2$ other than the ones
arising from Proposition \ref{linear}? In particular, are there any
examples with $h^-_J \geq 3$?
\end{question}

\vspace{0.3cm}

\noindent Theorem \ref{hJ-} follows from
Theorem \ref{nontrivK} and Propositions \ref{hyperelliptic}, \ref{kodsurface},
\ref{linear}. $\Box$

\vspace{0.2cm}

\subsection{Applications of Theorem \ref{hJ-}.}

We end this section with a couple of applications of our main
result. First, we prove that the $C^{\infty}$-pure property no
longer holds even for a K\"ahler $J$, if one
gives up the compactness of the manifold.

\begin{theorem} \label{nonpure}
Let $(M,J)$ be a compact complex surface with non-trivial canonical
bundle and non-zero geometric genus (equivalently, $h^-_J \neq 0$).
Let $B$ be a small contractible open set in $M$.
Then the $C^{\infty}$-full property for $J$ still holds on $M \setminus B$,
but the $C^{\infty}$-pure property for $J$ on $M
\setminus B$ no longer holds.
\end{theorem}

\begin{proof}  Since now is
not obvious to which set the groups
$H^{\pm}_J$ refer to, we'll use here the notations $H^{\pm}_J(M)$,
$H^{\pm}_J(M\setminus B)$, etc. By Mayer-Vietoris, the inclusion $i : M \setminus B
\hookrightarrow M$ induces an isomorphism in cohomology
$$H^2(M; \mathbb{R}) \stackrel{i^{*}}{\rightarrow} H^2(M\setminus B;
\mathbb{R}) \;. $$ Via this isomorphism, the subgroups $H^{\pm}_J(M)$
inject in $H^{\pm}_J(M \setminus B)$,  respectively. Thus, $(M
\setminus B, J)$ still has the $C^{\infty}$-full property.

\vspace{0.1cm}

For the $C^{\infty}$-purity statement, let $\alpha \in \mathcal{Z}^-_J(M)$, $\alpha \not\equiv 0$.
Choose a $J$-compatible metric $g$ and
a smooth function $r \geq 0$ compactly supported on $B$, so
that $r^2 |\alpha|_g^2 < 2$. Let $f = (1 - \frac{1}{2} r^2
|\alpha|_g^2)^{1/2}$ and let $\tilde J$ be the almost complex
structure defined by $g$ and $\tilde \omega = f \omega + r \alpha$ as in
(\ref{om(f,h,alpha)}). From Theorem \ref{nontrivK}, we have
 $H^-_{\tilde J}(M) = Span\{ [J\alpha]\}$.

Consider now the cohomology class $[\alpha]$. By Theorem
\ref{pf-dim4}, $[\alpha] \in H^+_{\tilde J}(M)$. Thus, there exists
a 1-form $\rho$ on $M$ such that $\alpha + d\rho$ is ${\tilde
J}$-invariant. On the other hand, by construction, it is clear that
$\tilde J = J$ on $M \setminus B$. Thus, on $M \setminus B$, $\alpha
+ d\rho$ is $J$-invariant, while $\alpha$ is obviously
$J$-anti-invariant. Hence $i^*[\alpha] \in H^+_J(M\setminus B) \cap
H^-_J(M\setminus B)$.

But the argument works for any $\alpha \in \mathcal{Z}^-_J(M)$. Thus, we get
$$ i^*(H^-_J(M)) \subset H^+_J(M\setminus B) \; \; \mbox{ and } \; i^*(H^+_J(M)) \subset H^+_J(M\setminus B) \; , \mbox{ so }$$
$$i^*(H^2(M;\mathbb{R})) = H^2(M \setminus B ; \mathbb{R}) = H^+_J(M\setminus B) .$$
Therefore, we obtain
$$ H^+_J(M \setminus B) \cap H^-_J(M \setminus B) = H^-_J(M \setminus B) \; ,$$
and the right hand-side is non-empty, as it contains at least $i^*(H^-_J(M))$.
\end{proof}

\vspace{0.1cm}

Next, we show that our examples of non-integrable almost complex
structures with $h_{\tilde J}^-=2$ from Proposition \ref{linear}
cannot admit a smooth pseudoholomorphic blowup.
\begin{theorem}\label{K3T4}
Suppose $\tilde J$ is a non-integrable almost complex structure with
$h_{\tilde J}^-=2$ on a $K3$ surface (or on $T^4$) and assume also
that ${\tilde J}$ is metric related to a complex structure. Then
there is no \emph{smooth} almost complex structure ${\tilde J}'$ on
$K3\#\overline{\mathbb{CP}^2}$ (or on
$T^4\#\overline{\mathbb{CP}^2}$) so that the blowup map
$f:K3\#\overline{\mathbb{CP}^2} \rightarrow K3$ (or
$f:T^4\#\overline{\mathbb{CP}^2}\rightarrow T^4$) is a $({\tilde
J}',\tilde J)$ holomorphic map. In other words, there is no
pseudoholomorphic blowup for such a $\tilde J$.
\end{theorem}
\begin{proof}
If there is such a ${\tilde J}'$, it should satisfy:
\begin{enumerate}
\item ${\tilde J}'$ is not integrable;
\item  $h_{\tilde J'}^-=2$;
\item $\tilde J'$ is metric related to a complex structure.
\end{enumerate}

\noindent However, by our Theorem \ref{hJ-}, there are no such almost complex
structures on $K3\#\overline{\mathbb{CP}^2}$ (or
$T^4\#\overline{\mathbb{CP}^2}$).
\end{proof}

\noindent The above proposition should be compared with Usher's result
\cite{Usher}: there is always such a Lipschitz continuous almost
complex structure $J'$.
The same argument but with some modification
of our previous definition can ensure that there is no such $C^1$
almost complex structure.

\section{Well-balanced almost Hermitian 4-manifolds}

In this section we introduce a class of 4-dimensional almost Hermitian structures
that contains the Hermitian ones and the almost K\"ahler ones.

\vspace{0.2cm}

\subsection{The image of the Nijenhuis tensor}

Given an almost complex structure $J$, at each point $p\in M$
define the image of its Nijenhuis tensor $N_J$
by
\[ Im(N_J)_p = {\rm Span } \{ N_J(X,Y) \; | X, Y \in T_p M \}. \]
This is $J$-invariant, that is if $Z \in Im(N_J)_p$, then $JZ \in Im(N_J)_p$.
The specific of dimension 4 is that at each point $Im(N_J)_p$ is either 0, or 2-dimensional,
but never 4-dimensional. This is so, because $N_J$ can be seen as a map
\[ N_J : T^{2,0}_J \rightarrow T^{0,1}_J, \; \; N_J(Z_1 \wedge Z_2) =
N_J(Z_1, Z_2) = [Z_1, Z_2]^{0,1}, \; Z_1,Z_2 \in T^{1,0}_J, \]
and in dimension 4 the bundle $T^{2,0}_J$  is real 2-dimensional. Here the superscripts denote
the usual complex type of vectors and forms induced by $J$.

One can ask when is $Im(N_J)$ a distribution over $M$.
This certainly happens when $J$ is integrable, as by Nirenberg-Newlander theorem
this holds if and only if $N_J = 0$ everywhere. To ask that $Im(N_J)$ is everywhere
2-dimensional on $M$ is equivalent to say that $N_J$ is non-vanishing at each point.
As the Nijenhuis tensor can be seen as a section of the bundle
$\Lambda^{2,0}_J \otimes T^{0,1}_J$, John Armstrong observed (\cite{Arm}, Lemma 3)
that the non-vanishing of $N_J$ at each point has topological consequences.
\begin{prop} (\cite{Arm}) \label{nonzeroN}
If $(M, J)$ is a 4-dimensional compact almost complex manifold with $N_J$
non-vanishing at each point, then the signature and Euler characteristic of $M$ satisfy
\[ 5\chi(M) + 6 \sigma(M) = 0 \; .\]
\end{prop}

\subsection{The well-balanced condition}

The following is a classical result (see, for instance, \cite{KN})
\begin{prop}
Let $(M, g, J, \omega)$ be an almost Hermitian manifold. Then
\begin{equation} \label{nablaom-N}
(\nabla_X \omega)(\cdot,\cdot) = 2<N_J(\cdot,\cdot), JX> +
\frac{1}{2} \Big( d\omega(X, \cdot, \cdot) - d\omega(X, J\cdot,
J\cdot) \Big)
\end{equation}
\end{prop}
\noindent It is well known that in dimension 4, there are just two Gray-Hervella
\cite{GrHer}
classes of special almost Hermitian manifolds -- Hermitian and
almost K\"ahler ones. These correspond to the vanishing (for any
$X$) of the first, respectively second term on the right side of
(\ref{nablaom-N}). In fact, on a general 4-dimensional almost
Hermitian manifold, let $\theta$ be the Lee form defined be $d\omega
= \theta \wedge \omega$. Then a short computation shows that
\begin{equation} \label{JXwedgetheta}
\frac{1}{2} \Big( d\omega(X, \cdot, \cdot) -  d\omega(X, J\cdot,
J\cdot) \Big) = ((JX)^{\flat} \wedge \theta)'',
\end{equation}
where the superscript $''$ denotes the $J$-anti-invariant part of a 2-form.
It is clear that the right hand-side of (\ref{JXwedgetheta}) vanishes for all
$X$ if and only if $\theta = 0$, i.e. $d\omega = 0$.

Relaxing both the Hermitian and the almost K\"ahler conditions, it is natural to ask
that for every $X$ at least one (but not necessarily the same) of the terms in the right
hand-side of (\ref{nablaom-N}) vanishes. From the observations above, we know that in dimension 4
the Nijenhuis term vanishes for at least a two dimensional space at each point.
The proof of the following proposition is tedious (but
straightforward), so we just sketch it, leaving the interested
reader to fill in remaining details.

\begin{prop} \label{equiv-wellbal}
Let $(M^4, g, J, \omega)$ be a 4-dimensional almost Hermitian manifold.
Then the following statements are equivalent:

(i) For any $p \in M$ and any $X \in T_p M$, at least one of the
terms in the right side of (\ref{nablaom-N}) vanishes;

(ii) For any $p \in M$, $(N_J)_p = 0$, or $\theta_p^{\sharp} \in Im(N)_p$;

(iii) $ \Big( \imath_{N_J(X,Y)} d\omega \Big)'' = 0 $,  for any $X,Y
\in T_p M$ and $p \in M$;

(iv) For any local non-vanishing section $\psi \in \Omega_J^-$,
\[ |\nabla \psi|^2 = |\nabla (J\psi)|^2 , \; \; <\nabla \psi, \nabla (J\psi)> = 0 . \]
\end{prop}

\begin{proof} Any (smooth) local section $\phi \in \Omega^-_J$ with $|\phi|^2 =
2$, determines (smooth) local 1-forms $a, b, c$ by
\begin{equation} \label{abc}
\begin{array}{ll}
\nabla \omega  &=  a \otimes \phi + b \otimes J \phi \\
\nabla \phi    &=  - a \otimes \omega + c \otimes J \phi \\
\nabla (J \phi) &=  - b \otimes \omega - c \otimes \phi ,
\end{array}
\end{equation}
We show that conditions (i), (ii), (iii), (iv) are all equivalent
with:

(v) For any point $p \in M$, there exists an open set $U$ containing
$p$ and a section $\phi \in \Omega^-_J$, defined on $U$, with
$|\phi|^2 = 2$ , so that the corresponding 1-forms $a$ and $b$
satisfy the pointwise conditions
\begin{equation} \label{wellbal1}
|a|^2 = |b|^2 \mbox{ and }  \; <a,b> = 0.
\end{equation}

Note first of all that if the condition (\ref{wellbal1}) is
satisfied for a given section $\phi \in \Omega^-_J$ with $|\phi|^2 =
2$, then it holds for any other section $\tilde \phi$ with the same
property (in other words, (\ref{wellbal1}) is ``gauge''
independent). Indeed, let
\[ \tilde \phi = \cos t \; \phi + \sin t \; J\phi, \]
for some smooth local function $t$. The corresponding 1-forms
given by (\ref{abc}) change as
\begin{equation} \label{abc-change}
\begin{array}{ll}
\tilde a & = a \, \cos t + b \, \sin t \\
\tilde b & = - a \, \sin t + b \, \cos t \\
\tilde c & = c + d t .
\end{array}
\end{equation}
Then it is easily checked that $\tilde a, \tilde b, \tilde c$
satisfy (\ref{wellbal1}), assuming that $a, b, c$ did so.

We prove now the equivalence (iv) $ \Leftrightarrow $ (v). Given a
section $\phi \in \Omega^-_J$, with $|\phi|^2 = 2$ and the 1-forms
$a, b, c$ defined by (\ref{abc}), one checks that
\[ |\nabla \phi|^2 - |\nabla J \phi |^2 = 2(|a|^2 - |b|^2) \; , \;
\; \; <\nabla \phi, \nabla J\phi> = 2<a, b> .\] Hence, the
implication (iv) $\Rightarrow$ (v) is clear. For the other
implication, let $\psi \in \Omega_J^-$ be a local non-vanishing
section and let $\phi = \frac{\sqrt{2} \psi}{|\psi|}$.
Straightforward computations imply
\[ |\nabla \phi|^2 - |\nabla J\phi|^2 = \frac{2(|\nabla \psi|^2 - |\nabla
J\psi|^2)}{|\psi|^2} \; , \; \; <\nabla \phi, \nabla J \phi> =
\frac{2 <\nabla \psi, \nabla J \psi>}{|\psi|^2}, \] and (v)
$\Rightarrow$ (iv) follows now easily.

Using (\ref{JXwedgetheta}), the reader can check the equivalences
(i) $ \Leftrightarrow $ (ii) $ \Leftrightarrow $ (iii). We will show
next that (i) $ \Leftrightarrow $ (v). Let $\phi$ be a local section
in $\Omega^-_J$, with $|\phi|^2 = 2$. Then, using the symmetries of
the Nijenhuis tensor and (\ref{JXwedgetheta}) one can check that
\[ 2<N_J(\cdot,\cdot), JX> = m(X) \otimes \phi - Jm(X) \otimes J\phi ,\]
\[ \frac{1}{2} \Big( d\omega(X, \cdot, \cdot) -  d\omega(X, J\cdot,
J\cdot) \Big) = n(X) \otimes \phi + Jn(X) \otimes J\phi, \] where
$m$ and $n$ are local 1-forms. Thus, with respect to the chosen
section $\phi$, the 1-forms $a, b$ given by (\ref{abc}) are given by
\[ a = m + n \; , \; \; b = -Jm + Jn .\]
Easy computation shows that (\ref{wellbal1}) is equivalent to
\[ <m, n> = 0 \; , \; \; <m, Jn> = 0 , \] which is easily seen to be
equivalent to (i).

\end{proof}

\begin{definition} (i) An almost Hermitian manifold $(M^4, g,
J, \omega)$ is called {\it well-balanced} if it satisfies one (and hence
all) of the conditions in Proposition \ref{equiv-wellbal}.

(ii) An almost complex structure $J$ on a 4-manifold $M^4$ is called
{\it well-balanced} if it admits a compatible well-balanced almost
Hermitian structure.
\end{definition}

\noindent It is interesting to understand how large is the class of
well-balanced almost complex structures on compact 4-manifolds, but
we leave this problem for future study. Locally, any almost complex
structure in dimension 4 is compatible with some symplectic form
\cite{Lej}, so locally any almost complex structure is
well-balanced.

\vspace{0.2cm}

The following result provides examples of
well-balanced almost Hermitian 4-manifolds.

\begin{prop} \label{expwb} (i) Any 4-dimensional Hermitian or almost K\"ahler manifold is
well-balanced.

(ii) Suppose $g$ is a Riemannian metric adapted to a
complex-symplectic on a 4-manifold; in other words, assume that $g$
is compatible to a triple $I, J, K$ of almost complex structures
satisfying the quaternion relations and assume that $I$ is
integrable, and that $(g,J)$ and $(g, K)$ are almost K\"ahler. Then
for any constant angles $t$ and $s$, the almost Hermitian structure
$(g, \tilde J)$ with $\tilde J = \cos t \, I + \sin t \, (\cos s \,
J + \sin s \, K) $ is well-balanced.

(iii) Let $M$ be a compact quotient by a discrete subgroup
 of the 3-step nilpotent Lie group $G$, whose nilpotent Lie algebra $\mathfrak{g}$
has structure equations
\[ de^1 = de^2 = 0, de^3 = e^{1} \wedge e^{4}, de^4 = e^{1} \wedge e^{2} .\]
Consider the invariant metric $g = \sum (e^i \otimes e^i) $ and the compatible almost complex structure
$J$ given by $Je^1 = e^2$, $Je^3 = e^4$. Then $(g,J)$ is well-balanced.
\end{prop}
\begin{proof}
(i) In either case, it is obvious that condition (iii) of Proposition \ref{equiv-wellbal} is satisfied.

\vspace{0.2cm}

(ii) It is clear that it is enough to check the case $s=0$. Let us
denote $\omega_I, \omega_J, \omega_K$ the three fundamental forms.
Since $(g, I)$ is Hermitian and $(g,J), (g,K)$ are almost K\"ahler,
we have
\begin{equation} \label{aIJK}
\begin{array}{ll}
\nabla \omega_I  &=  a \otimes \omega_J + Ia \otimes \omega_K \\
\nabla \omega_J    &=  - a \otimes \omega_I - Ja \otimes \omega_K \\
\nabla \omega_K &=  - Ia \otimes \omega_I + Ja \otimes \omega_J ,
\end{array}
\end{equation}
for a 1-form $a$. Let $\tilde \omega$ the form corresponding to
$\tilde J = \cos t I + \sin t J$. Taking $\tilde \phi = \omega_K$, a
short computation shows that
\[ \nabla \tilde \omega = (Ia \cos t - Ja \sin t) \otimes \tilde
\phi - a \otimes \tilde J \tilde \phi ,\] and the statement is
easily verified.

(iii) Direct computation shows that at each point
$Im(N_J) = {\rm Span}(e_3, e_4)$. Even without computation, one can verify this
by noting that the commutator $[\mathfrak{g}, \mathfrak{g}]$ is ${\rm Span}(e_3, e_4)$
and this is $J$-invariant, by the definition of $J$. Next, using the structure equations,
one checks that $d\omega = - e^3 \wedge \omega$, where $\omega = e^1 \wedge e^2 + e^3 \wedge e^4$.
Thus, condition (ii) of Proposition \ref{equiv-wellbal} is satisfied.
\end{proof}

\begin{remark}
{\rm Note that $J$ from example (iii) in Proposition \ref{expwb}
is not integrable and cannot be tamed by a symplectic form on $M$
(see Proposition 6.1 in \cite{FT} and Remark \ref{expFT} above).}
\end{remark}

\noindent To state the main result of this subsection we need one more
definition.
\begin{definition} An almost Hermitian manifold $(M^4, g,
J)$ has {\it Hermitian type Weyl tensor} if
\begin{equation} \label{HW}
<W^+(J\beta), J\beta> = <W^+(\beta), \beta>, \mbox{ for any } \beta
\in \Omega^-_J .
\end{equation}
\end{definition}


\noindent It is well known that if $J$ is integrable (i.e.
$(M^4, g, J)$ is a Hermitian manifold), then (\ref{HW}) holds.
Also, any almost Hermitian structure with an ASD metric trivially
satisfies (\ref{HW}).

\begin{theorem} \label{mainwb}
Let $(M^4,J)$ be a compact almost complex 4-manifold which admits a
compatible Riemannian metric $g$ so that $(g,J)$ is well-balanced
and has Hermitian type Weyl tensor. Then $h^-_J = 0$ or $J$ is
integrable.
\end{theorem}

\begin{proof} Suppose $\beta$ is a non-trivial closed, $J$-anti-invariant
form on $M$. The next Lemma shows that, under the given assumptions,
$J\beta$ is also closed, thus $\Phi = \beta + i J\beta$ is a closed,
complex form of (2,0) type. The integrability of $J$ then follows
(see e.g. \cite{Sal}).

\begin{lemma}
Let $(M^4, g, J, \omega)$ be a compact, almost Hermitian 4-manifold
which is well-balanced and has Hermitian type Weyl tensor. Then for
any $\beta \in \Omega^-_J$, $d\beta = 0 \Leftrightarrow d(J\beta) =
0$.
\end{lemma}

\noindent {\it Proof of Lemma:} It's enough to prove  $d \beta = 0
\Rightarrow d (J\beta) = 0$. The well-known Weitzenb\"ock formula
for a 2-form $\psi$ is
\[ \int_M ( |d \psi|^2 + |\delta \psi|^2 - |\nabla \psi|^2 ) \; dV
= \int_M (\frac{s}{3} |\psi|^2 - <W(\psi), \psi>) \; dV . \]
Applying this for $\beta$ and $J\beta$ and using the assumption on
the Weyl tensor, we get
\[ \int_M ( |d \beta|^2 + |\delta \beta|^2 - |\nabla \beta|^2 ) \; dV
 = \int_M ( |d (J\beta)|^2 + |\delta (J\beta)|^2 - |\nabla(J\beta)|^2 ) \;
dV .\] Now, by assumption $\beta \in \Omega^-_J$ and  $d\beta = 0$,
thus $\beta$ is harmonic, so it is non-vanishing on an open dense
set in $M$. From the well-balanced assumption and continuity, we get
that $|\nabla(J\beta)|^2 = |\nabla \beta|^2$ everywhere on $M$.
Thus,
\[ 0 = \int_M ( |d \beta|^2 + |\delta \beta|^2 ) \; dV
 = \int_M ( |d (J\beta)|^2 + |\delta (J\beta)|^2 ) \;
dV .\] The lemma and the Theorem are thus proved.
\end{proof}

\noindent The following is an immediate consequence.
\begin{cor}
A compact 4-dimensional almost K\"ahler structure $(g,J, \omega)$
with Hermitian type Weyl tensor and with $h^-_J \neq 0$ must be
K\"ahler.
\end{cor}

\begin{remark} {\rm Under different additional conditions, some other
integrability results have been obtained for compact, 4-dimensional
almost K\"ahler manifolds $(g,J, \omega)$ with Hermitian type Weyl
tensor (see \cite{aa}, \cite{aad}).}
\end{remark}

\begin{remark} {\rm The corollary implies that if we start with a K\"ahler
surface $(M, g, J, \omega)$
and define the almost complex structures $\tilde J^{\pm}_{\alpha}$
corresponding to (\ref{om(f,h,alpha)}) and (\ref{hf}) for $\alpha
\in \mathcal Z^-_J$, then $\tilde J^{\pm}_{\alpha}$ cannot admit compatible
almost K\"ahler structures with Hermitian-type Weyl tensor.
They do admit compatible almost K\"ahler
structures (since $\pm \omega + \alpha$ is symplectic).}
\end{remark}

\section{Symplectic Calabi-Yau equation and semi-continuity property of $h_J^{\pm}$} \label{sCY}

In this section, we  use the beautiful ideas in \cite{D} to
establish a stronger semi-continuity property for $h_J^{\pm}$ than
in Theorem \ref{semicont-hJt}, near an almost complex structure
which admits a compatible symplectic form.

\subsection{Symplectic CY equation and openness}
The classical Calabi-Yau theorem can be stated as follows: Let
$(M,J, \widetilde{\omega})$ be a K\"ahler manifold. For any volume form $\sigma$
satisfying $\int_M\sigma=\int_M\widetilde {\omega}^n$, there exists a unique
K\"ahler form ${\omega}$ with
$[\omega]=[\widetilde{\omega}]$ s.t. ${\omega}^n=\sigma$.

Yau's original proof of the existence (\cite{Y}) makes use of a
continuity method between the prescribed volume form $\sigma$ and
the natural volume form $\widetilde{\omega}^n$. The proof of openness is by the
implicit function theorem. The closedness part is obtained by a
priori estimates.

\subsubsection{Set up}In \cite{D}, Donaldson
introduced the symplectic version of the Calabi-Yau equation.

 Let $(M, J)$ be a compact almost  complex
$2n-$manifold and assume that $\Omega$ is a symplectic form
compatible with $J$.
 For any function $F$ with
\begin{equation}
\int_M e^F \widetilde{\omega}^n=\int_M \widetilde{\omega}^n
\end{equation}
the symplectic CY equation is the following equation of a
$J-$compatible symplectic form $\widetilde \omega$,
\begin{equation} \label{syCY}  \omega^n=e^F \widetilde{\omega}^n.
\end{equation}

In \cite{D}, Donaldson further observed that solvability of the
symplectic CY equation  in dimension $4$ may lead to some amazing
results in four dimensional symplectic geometry.

\subsubsection{Openness}
In Donaldson's paper, he proves that the solution set of the
symplectic CY equation (\ref{syCY})  is open by using the implicit
function theorem. This only works for dimension $4$. Donaldson
actually works in the  general setting of $2-$forms on $4$
manifolds.

Suppose $M$ is a $4-$manifold with a volume form $\rho$ and a choice
of almost-complex structure $J$. At any point $x$, $\rho$ and $J$
induce a volume form and a complex structure on the vector space
$T_x(M)$. Denote by $P_x$  the set of positive $(1,1)-$forms whose
square is the given volume form. Then $P_x$ is a three-dimensional
submanifold in $\Lambda^2T_x(M)$ (a sphere in a $(3,1)-$space).  We
consider the $7-$dimensional manifold $\mathcal  {P}$ fibred over
$M$ with fiber $P_x$,
\[\mathcal  {P}=\{\omega^2=\rho\,|\, \omega \hbox{ is compatible with
$J$} \}.\] It is a submanifold of the total space of the bundle
$\Lambda^2$.

Now, we want to find a symplectic form $\omega$ which is compatible
with $J$ and has fixed volume form with some cohomology conditions.
That is, we are searching for $\omega$ satisfying the following
conditions (we call this condition  type $D$):
\begin{equation}
\left\{ \begin{array}{ll}
          \omega & \subset \mathcal{P}_{\rho}, \\
          d\omega & =0, \\
           \left[\omega\right] & \in e + H^2_+ \subset H^2(M;\mathbb{R}).
\end{array} \right.
\end{equation}
 Here $e$ is a fixed cohomology
class and $H^2_+$ is a maximal positive subspace. Notice, we have
three families of variables: $\rho$, $J$ and $e$. In particular, $e$
varies in a finite dimensional space.

We have the following result which is a slight variation of
Proposition $1$ in \cite{D}.

\begin{prop} \label{sol}Suppose $\omega$ is a solution of type $D$
constrain with given $\mathcal{P}$ and $e$. If we have a smooth
family $\mathcal{P}^{(b)}$  parameterized by a Banach space $B$,
$\{\mathcal{P}^{(b)}\}$, with $\mathcal{P}=\mathcal{P}^{(0)}$ and
$b$ varies in $B$, then we have a unique solution of the deformed
constraint in a sufficiently small neighborhood of $0$ in $B$.
Further, this solution lies in a small  $C^0$ neighborhood  of $\omega$.
\end{prop}

We just indicate
how to find a small neighborhood for which we have the existence.

For each point $x\in M$, the tangent space to ${ P}_x$ at
$\omega(x)$ is a maximal negative space. Thus the solution $\omega$
determines a conformal structure on $M$. We fix a Riemannian metric
$g$ in this conformal class (actually, we can choose the metric
determined by $\omega$ and $J$). For small $\eta$, $\omega+\eta$
lying in $\mathcal P_{\rho}$ is expressed as
\[\eta^+=Q(\eta),\]
where $Q$ is a smooth map with $Q(\eta)=O(\eta^2)$. After choosing
$2-$form representatives of $H^2_+$, closed forms $\omega+\eta$
satisfying our cohomological constraint can be expressed as
$\omega+da+h$ where $h\in H^2_+$ and where $a$ is a $1-$form
satisfying the gauge fixing constraint $d^*a=0$. Thus our
constraints correspond to the solutions of the PDE
\begin{equation} \label{PDE} \left\{ \begin{array}{ll} d^*a&=0 \\
d^+a&=Q(da+h)-h^+.
\end{array} \right.
\end{equation}

Thus, our constraints are represented by a system of nonlinear elliptic PDE.
Donaldson further observes that its linearization
\[L=d^*\oplus d^+:\Omega^1/\mathcal{H}^1\longrightarrow
\Omega^0/\mathcal{H}^0 \oplus \Omega^2_+/H^2_+\] is invertible. Then
we apply the following version of the implicit function theorem:

\begin{theorem} \label{implicit} Let $X$, $Y$, $Z$ be Banach spaces and  $f: X \times Y
\longrightarrow Z$ a Fr\'echet differentiable map. If $(x_0,y_0)\in
X \times Y$, $f(x_0,y_0) = 0$, and $y\mapsto D_2f(x_0,y_0)(0,y)$ is a
Banach space isomorphism from $Y$ onto $Z$. Then there exist
neighborhoods $U$ of $x_0$ and $V$ of $y_0$ and a Fr\'echet
differentiable function such that $f(x,g(x)) = 0$ and $f(x,y) = 0$
if and only if $y = g(x)$, for all $(x,y)\in X \times Y$.
\end{theorem}

To use this theorem, first notice that $D_2f$ is just our $L$
defined above, which is invertible at a solution of our constraints.
Moreover, $X$ is our parametrization space $B$, $Y$ is
$(\Omega^{1})_1/\mathcal{H}^1$, $Z$ is $(\Omega^{0})_0/\mathcal{H}^0
\oplus (\Omega^{2}_+)_0/H^2_+$. Here $(\Omega^{n})_m$ represents the
space of  $C^m$ $n-$forms.

Then every condition is satisfied in our
setting.

\subsection{Semi-continuity properties of $h_J^{\pm}$}

\subsubsection{Weak and strong neighborhoods}
\label{ngh}

As described in \ref{topJs}, the  space of $C^{\infty}$ almost complex structures
$\mathcal{J}=\mathcal{J}^\infty$ is not a Banach manifold but a
Fr\'echet manifold. In this case we can still apply Proposition
\ref{sol} to a smooth path or a finite dimensional space (hence
Banach) in $\mathcal J$.  That is to say, if an almost complex
structure $J$ has a solution of the CY equation
${\omega}^2=\rho$ with a $J-$compatible form $\omega$ satisfying
$[{\omega}]\in e+H_+^2$, then for any path through $J$,
there is a small interval near $J$ such that the CY type
equation is solvable with conditions in $(D_t)$ in this interval. In
the end we get a weak neighborhood--the union of all the intervals.
Notice that this is not necessarily ``a small ball" near $J$, i.e.
it may not have an interior point.

We would like to apply Proposition 5.2 to an open neighborhood with respect to
the $C^{\infty}$ topology, which can be called a strong neighborhood compared with
the one described above.
For this purpose, notice that the tangent space $T_J\mathcal{J}^l$ at $J$
consists of $C^l-$sections $A$ of the bundle $End(TM,J)$ such that
$AJ+JA=0$.
It is a Banach space with $C^l$ norm. Moreover, this gives rise to a
local model for $\mathcal J^l$ via $Y\longmapsto J\rm{exp}(-JY)$.
Thus we can apply Proposition \ref{sol} to a Banach chart of  $J$ in
the space of $C^l$ almost complex structures $\mathcal{J}^l$ endowed
with $C^l$ norm.

\begin{cor} \label{nbd} If we parameterize $\mathcal{P}^{(b)}$ in
Proposition \ref{sol} by a neighborhood $U(J_0)$ of $J_0$ in
$\mathcal{J}$ with $C^{\infty}$ topology, then we can get a
small neighborhood of $J$ satisfying all the properties stated in
Proposition \ref{sol} under the same topology.
\end{cor}

\begin{proof}
The space of $C^1$ almost complex structures $\mathcal J^1$ with
$C^1$ norm is a Banach manifold. We parameterize a neighborhood of
$J_0$ by an open set in the induced Banach space.

Then we can apply Proposition \ref{sol} for this setting. The only
point we need to check is that the Fr\'echet differentiability of
the reliance of our constraints with respect to the parametrization
space $\mathcal J^1$. Here, we adapt the arguments in \cite{W}.
We define a tensor $\Pi$ (which is denoted by $\mathcal{P}$ in \cite{W}) as
\[\Pi^{ij}_{kl}=\frac{1}{2}(\delta^i_k\delta^j_l-J^i_kJ^j_l).\]
When restricting $\Pi$ on the space of $2-$forms, it is just the
projection onto the $J-$anti-invariant part. As in \cite{W}, we also
define $\chi_1, \cdots, \chi_r$ be self-dual harmonic $2-$forms with
respect to $\omega$ such that $\{\omega, \chi_1, \cdots, \chi_r\}$
are $L^2$ orthogonal bases for $\mathcal H^+_{\omega}$.

Consider the operator $\Phi: (\Omega^{1})_1\times \mathbb R^r\times
\mathcal J^1\longrightarrow (\Omega^2)_0$ by

\[\Phi(b,\underline{s},J)=(\log\dfrac{(\omega+\sum_{i=1}^rs_i \chi_i+db)^2}{\omega^2})\frac{(Id-\Pi_J)\omega}{2}+
\Pi_J(\omega+\sum_{i=1}^rs_i \chi_i+db),\]

The solution of $\Phi(b,\underline{s},J)=0$ gives a closed,
$J-$invariant form with the same volume form as $\omega$. In other
words, we get a description of our constraints by the zero set of a
map. It is easy to see that the map $\Phi$ is a Fr\'echet
differentiable map.

Thus by Proposition \ref{sol} we have a neighborhood $U^1(J_0)$ of
$J_0$ in which we have all the properties stated there. Especially,
we can suppose $U^1(J_0)$ is a ball in $\mathcal J$ with radius
$\epsilon$.

Finally, the small neighborhood of $J_0$ in $C^{\infty}$ with
$d(J_0,J)<\frac{\epsilon}{2(1+\epsilon)}$, where $d$ is defined in \eqref{metric-top},  is what we want.
\end{proof}

\subsubsection{Variations of $h_J^{\pm}$}

Following \cite{LZ}, given a compact almost complex manifold $(M,J)$ define
the $J$-{\em compatible symplectic cone}
$$
\mathcal{K}^c_J=\left\{[\omega]\in
H^2_{dR}(M;\mathbb{R})\,\,\,\vert\,\,  \hbox{\rm $\omega$ symplectic and }\,J\,\, \hbox{\rm is}\,\,
\omega\hbox{-\rm compatible}\right\}\,.
$$
It is easy to see that $\mathcal{K}^c_J$ is an open convex set in $H^+_J$.
It is also immediate from the definition that $\mathcal{K}^c_J \neq \emptyset$ if and only if
$J$ admits compatible symplectic forms.

\begin{theorem} \label{deformation} Suppose $M$ is a $4-$manifold with an
almost complex structure $J$ such that $\mathcal K_J^c(M) \neq \emptyset$. is
non-empty. Then for any almost complex structure $J'$ in a
sufficiently small neighborhood of $J$ as in Corollary \ref{nbd}, we
have
\begin{itemize}
 \item $\mathcal K_{J'}^c(M)\neq \emptyset$;

  \item $h_J^{+}(M)\le h_{J'}^{+}(M)$;

   \item $h_J^{-}(M)\ge h_{J'}^{-}(M)$.
\end{itemize}
\end{theorem}

\begin{proof} The first statement is a direct consequence of  Corollary \ref{nbd},
and was already observed by Donaldson
 (see also \cite{Lej}).

As  $\mathcal K_{J'}^c(M)$ and $\mathcal K_J^c(M)$ are nonempty open
sets  in $H_{J'}^{+}(M)$ and $H_J^{+}(M)$ respectively, to estimate
$h_J^{+}(M)$ and $h_{J'}^{+}(M)$, we only need to estimate the
dimensions of $\mathcal K_{J'}^c(M)$ and $\mathcal K_J^c(M)$.

Let $h=h_J^{+}(M)$. We choose $h$ rays which are ``in general
position", i.e. the interior of their span is an open set of
$\mathcal K_J^c(M)$. We suppose the $h$ rays are $C\cdot
[\omega_i]$'s where $\omega_i$'s are the $J-$compatible forms and
$[\omega_i]$'s have homology norm 1 with respect to some bases.

Then we use Corollary \ref{nbd} for each $i$ with fixed volume
form $\omega^2_i$. Then we have $h$ neighborhoods $U_i$ such that
for $J'\in U_i$, we have a $J'$ compatible form $\omega'_i$ which is
a small perturbation of $\omega_i$. Let $U$ be the intersection of
these $h$ neighborhoods. Then for any $J' \in U$, we have
$\omega'_i$'s which are still in the general position (because they
are perturbed in $C_0$ norm from a general position). And we see that the span of
the $h$ new rays belongs to $\mathcal K_{J'}^c(M)$ because
positive combinations of $\omega'_i$'s are still $J'-$compatible
forms. Hence we have $h_J^{+}(M)\le h_{J'}^{+}(M)$.
The last inequality  is a consequence of the previous one and Theorem
\ref{pf-dim4}.
\end{proof}

\begin{remark}
{\rm The first statement also means that,  on a $4-$manifold, the space of
almost K\"ahler complex structure $\mathcal J_{ak}$ is an open
subset of $\mathcal J$. If one considers complex deformation, the analogue of the first statement is  a classical
theorem of Kodaira and Spencer. Their theorem is in fact valid for
any even dimension.}
\end{remark}

Let us consider the stratification
 \[\mathcal J=\bigsqcup_{i=0}^{b^+} \mathcal
J_i,\] where $J\in \mathcal J_i$ if $h_J^-=i$. Then we have

\begin{cor} \label{0ak} On a $4-$manifold, $\mathcal J_0\cap \mathcal J_{ak}$ is
open in $\mathcal J$.
\end{cor}

\noindent It is known that  $\mathcal J_{ak}$ is never the full space $\mathcal J$.
In fact, in any connected component of $\mathcal J$ there are non-tamed almost complex structures (see e.g. \cite{D}).
Nonetheless Corollary \ref{0ak} is a strong evidence of  Conjecture \ref{conj1}.
{\rm In addition, the path-wise semi-continuity
established in Theorem \ref{semicont-hJt} indicates that the strong
semi-continuity property of Theorem \ref{deformation} very likely
holds for every $J$. This would imply that $\mathcal J_0$ is open in
$\mathcal J$. }

\end{document}